\documentclass[reqno,oneside,12pt]{amsart}
\usepackage{fancyhdr}
\usepackage{xspace}
\usepackage{geometry}
\usepackage{tikz}
\usepackage{tikz-cd}
\usepackage{aligned-overset}
\tikzset{node distance=1.5cm, auto}
\usepackage{color}
\definecolor{darkgreen}{rgb}{0,0.45,0}
\usepackage[pagebackref,colorlinks,citecolor=darkgreen,linkcolor=darkgreen]{hyperref}
\usepackage[capitalize]{cleveref}
\usepackage{stmaryrd}
\usepackage{enumerate}
\usepackage[shortlabels]{enumitem}
\usepackage{makecell}
\usepackage[justification = centering]{caption}
\usepackage{cancel}

\usepackage{graphicx}
\usepackage{enumerate}
\usepackage[all]{xy}
\usepackage{amsmath,amsfonts,amssymb}
\usepackage{mathtools,thmtools}
\usepackage[latin9]{inputenc}

\newcommand{\End}{{\sf End}}

\newcommand{\Mod}{{\sf Mod}}

\newcommand{\Mat}{{\sf Mat}}

\newcommand{\PMod}{{\sf PMod}}

\renewcommand{\Bbbk}{k}

\newtheorem{prop}{Proposition}[section]
\newtheorem{proposition}[prop]{Proposition}
\newtheorem{lemma}[prop]{Lemma} 
 
\newtheorem{corollary}[prop]{Corollary} 

\newtheorem{theorem}[prop]{Theorem}

\newtheorem{alphthm}{Theorem}

\newtheorem{alphconj}[alphthm]{Conjecture}

\theoremstyle{definition}

\newtheorem{definition}[prop]{Definition}

\declaretheorem[name=Example,qed={$\lozenge$},sibling=proposition]{example}

\declaretheorem[name=Remark,qed={$\triangle$},sibling=proposition]{remark}

\newtheorem{construction}[prop]{Construction}

\newcommand{\benu}{\begin{enumerate}}
\newcommand{\enu}{\end{enumerate}}

\newcommand{\beqna}{\begin{eqnarray}}
\newcommand{\eqna}{\end{eqnarray}}
\newcommand{\beqnast}{\begin{eqnarray*}}
\newcommand{\eqnast}{\end{eqnarray*}}
\newcommand{\beqn}{\begin{equation}}
\newcommand{\eqn}{\end{equation}}
\newcommand{\beqnst}{\begin{equation*}}
\newcommand{\eqnst}{\end{equation*}}

\usepackage{latexsym}

\usepackage{xspace}
\usepackage{amscd}
\usepackage{amssymb}
\usepackage{amsfonts}

\makeatother

\newcommand{\bema}{\left ( \begin{array}}
\newcommand{\ema}{\end{array} \right )}

\DeclareMathOperator{\cotensor}{\square}

\begin{document}

\title{Towards a classification of simple partial comodules of Hopf algebras}

\author[E. Batista]{Eliezer Batista}
\address{Departamento de Matem\'atica, Universidade Federal de Santa Catarina, Brazil}
\email{ebatista@mtm.ufsc.br}
\author[W. Hautekiet]{William Hautekiet}
\address{D\'epartement de Math\'ematiques, Universit\'e Libre de Bruxelles, Belgium}
\email{william.hautekiet@ulb.be}
\author[P. Saracco]{Paolo Saracco}
\address{D\'epartement de Math\'ematiques, Universit\'e Libre de Bruxelles, Belgium}
\email{paolo.saracco@ulb.be}
\author[J. Vercruysse]{Joost Vercruysse}
\address{D\'epartement de Math\'ematiques, Universit\'e Libre de Bruxelles, Belgium}
\email{joost.vercruysse@ulb.be}

\thanks{\\ {\bf 2020 Mathematics Subject Classification}: 16T05, 16T15, 16D60. \\   {\bf Key words and phrases:} partial corepresentation, partial comodule.}

\flushbottom

\begin{abstract} 
	Making the first steps towards a classification of simple partial comodules, we give a general construction for partial comodules of a Hopf algebra \(H\) using central idempotents in right coideal subalgebras and show that any \(1\)-dimensional partial comodule is of that form. We conjecture that in fact all finite-dimensional simple partial \(H\)-comodules arise this way. For \(H = kG\) for some finite group \(G\), we give conditions for the constructed partial comodule to be simple, and we determine when two of them are isomorphic.  If \(H = kG^*,\) then our construction recovers the work of M. Dokuchaev and N. Zhukavets \cite{DZ}. We also study the partial modules and comodules of the non-commutative non-cocommutative Kac-Paljutkin algebra \(\mathcal{A}\).
\end{abstract}

\maketitle

\tableofcontents

\section*{Introduction}
\thispagestyle{empty} 
The theory of partial group actions and partial group representations is a relatively young field of research, that arose from the aim to describe the structure of certain $C^*$-algebras \cite{Exel} and whose systematic study was started in \cite{DE} and \cite{DEP}. Partial actions of Hopf algebras as initiated in \cite{CJ} then inspired the notion of partial representations of Hopf algebras \cite{ABVparreps}. The theory of partial representations has been shown to be a considerable enrichment of classical representation theory. Not only can many properties of classical (or: global) representations be extended to the partial setting (for example, the category of partial representations bears a closed monoidal structure), but in addition, partial representations allow to delve deeper into the internal algebraic structure of the considered group or Hopf algebra. Indeed, already in \cite{DEP}, it has been shown that partial representations of a finite group encode information about its lattice of subgroups. In fact, partial representations of a group coincide with usual representations of a suitably constructed groupoid. A similar phenomenon has been observed in \cite{ABVparreps}, showing that partial representations of a Hopf algebra correspond to modules over a suitably constructed Hopf algebroid. 

From the point of view of Tannaka duality, in the Hopf algebraic setting, it is more natural to study partial corepresentations, first introduced in \cite{ABQVcoreps}. In contrast to the case of partial representations, it was shown in \cite{BHV} that partial comodules are in general not equivalent to a category of usual comodules over any coalgebra, which indicates that the theory of partial corepresentations might experience additional subtleties and complications. The aim of this paper is to make some first steps towards a good understanding of simple partial comodules. 

We present a general construction (see \cref{construction:general2}) for partial comodules of a Hopf algebra \(H\), which is inspired by ideas contained in  \cite{ABenvelope} and \cite{DZ}. The construction is based on what we call a \textit{subcentral} idempotent, i.\,e.\ a nonzero idempotent \(e\) such that \(ee_{(1)} \otimes e_{(2)} = e_{(1)} e \otimes e_{(2)}\). These are central in the right coideal subalgebra \(A_e\) they generate, which induces, on its turn, a quotient coalgebra \(\bar{H}_e = H/HA_e^+\) (see \cref{le:correspondance_idealcoideal}). Taking a right \(\bar{H}_e\)-comodule \(W,\) we show that the cotensor product \(W \cotensor^{\bar{H}_e} He\) is a partial \(H\)-comodule. The construction reminds of the form of finite degree partial representations of groups presented in \cite[Corollary 2.4]{DZ}.
If we take \(H = kG^*\) for some finite group \(G\), then our construction recovers the one in \cite{DZ}, which implies that every simple partial comodule of \(kG^*\) is obtained by \cref{construction:general2}.

Keeping this in mind, we make the following conjectures about this construction of partial comodules.

\begin{alphconj}
    Let \(H\) be a Hopf algebra and \(e \in H\) a subcentral idempotent such that the right coideal subalgebra it generates is finite-dimensional. If \(W\) is a simple right \(\bar{H}_e\)-comodule, then \(W \cotensor^{\bar{H}_e} He\) is a simple partial \(H\)-comodule.
\end{alphconj}

\begin{alphconj}
    Let \(H\) be a Hopf algebra. Then every finite-dimensional simple partial \(H\)-comodule is of the form \(W \cotensor^{\bar{H}_e} H e\) for a subcentral idempotent \(e \in H\) and simple right \(\bar{H}_e\)-comodule \(W\). 
\end{alphconj}

Together with \cite[Conjecture 2.8]{ABVdilations}, which conjectures that the category of partial modules of a semisimple Hopf algebra is semisimple, these propose a classification of partial comodules of a finite-dimensional cosemisimple Hopf algebra \(H\), since \(H^*\) is then semisimple. Indeed, in that case the category of partial \(H\)-comodules, being isomorphic to the category of partial \(H^*\)-modules, is expected to be semisimple, so every partial \(H\)-comodule is a direct sum of simples. These simple partial comodules can be obtained by applying the construction, if \cref{conj:complete} holds true.  
In the rest of the article, we provide further evidence for the above conjectures, looking at particular cases and concrete examples.

In contrast to the global case, even the study of \(1\)-dimensional partial structures is already highly non-trivial, as one can also judge from the recent literature on this subject \cite{FMS, MPS}.
We show in \cref{se:1D} that \cref{conj:complete} holds for \(1\)-dimensional partial comodules of Hopf algebras with invertible antipode.

\setcounter{alphthm}{3}
\begin{alphthm}
    Every \(1\)-dimensional partial comodule over a Hopf algebra \(H\) with invertible antipode is of the form \(W \cotensor^{\bar{H}_e} He\) for some subcentral idempotent \(e \in H\) and some \(1\)-dimensional right \(\bar{H}_e\)-comodule \(W\).
\end{alphthm}

If \(H = kG\), for \(G\) a finite group, we are able to say a lot more about the constructed partial \(kG\)-comodules. This is done in \cref{se:groups}. We show in \cref{th:groupcase_simple} that whenever the \( \bar{H}_e\)-comodule \(W\) is taken simple, the resulting partial comodule will be simple too, hence proving \cref{conj:simple} for finite group algebras. Moreover, it turns out that two partial comodules obtained by \cref{construction:general2} are isomorphic exactly when they are related by a multiplicative character of a subgroup of \(G\) (\cref{th:redundancyI}). With the help of a computer program, we were able to determine the dimension of the partial Hopf algebras \((kS_3^*)_{par}, (kD_8^*)_{par}\) and \((kQ_8^*)_{par}\). This way we show that in fact our construction produces a complete list of simple partial comodules over these groups, and that the category of their partial comodules is semisimple, supporting the conjectures mentioned above.

The structure of the article is as follows. In \cref{se:prelim}, we recall some preliminary definitions and results about partial modules, partial comodules and right coideal subalgebras. We also provide some examples that will prove to be useful in the main part of the paper. 
In \cref{se:construction}, the general construction is outlined and we show that it recovers the known construction for simple partial modules of groups from \cite{DZ}. \cref{se:1D} is dedicated to the proof of \cref{th:1D}. The case \(H = kG\) is studied in \cref{se:groups}, and finally, in \cref{se:kac}, we study the partial comodules of the Kac-Paljutkin algebra \(\mathcal{A}\), which is the unique \(8\)-dimensional non-commutative non-cocommutative semisimple Hopf algebra. 

\section{Preliminaries}
\label{se:prelim}

Throughout this text, $\Bbbk$ will be the base field and $H$ will be a Hopf algebra. For the comultiplication we adopt the Sweedler notation \(\Delta(h) = h_{(1)} \otimes h_{(2)}\).

\subsection{Partial representations and partial modules}

\begin{definition}[{\cite{ABVparreps}}]
	A \emph{partial representation} of a Hopf algebra $H$ over a $\Bbbk$-algebra $B$ is a linear map $\pi :H \rightarrow B$ such that
	\begin{enumerate}
		\item[(PR1)] $\pi (1_H) =1_B$;
		\item[(PR2)] $\pi (h)\pi (k_{(1)}) \pi (S(k_{(2)}))=\pi (hk_{(1)}) \pi (S(k_{(2)}))$, for every $h,k\in H$;
		\item[(PR3)] $\pi (h_{(1)}) \pi (S(h_{(2)})) \pi (k) =\pi (h_{(1)}) \pi (S(h_{(2)}) k)$, for every $h,k\in H$.
	\end{enumerate}
\end{definition}

One can show that axioms (PR2) and (PR3) are equivalent to
\begin{enumerate}
	\item[(PR4)] $\pi (h)\pi (S(k_{(1)})) \pi (k_{(2)}) =
	\pi (h S(k_{(1)})) \pi (k_{(2)})$, for every $h,k\in H$;
	\item[(PR5)] $\pi (S(h_{(1)})) \pi (h_{(2)}) \pi (k) =\pi (S(h_{(1)})) \pi (h_{(2)} k)$, for every $h,k\in H$.
\end{enumerate}

In particular, if $B=\text{End}_\Bbbk (M)$ for a $\Bbbk$-vector space $M$, then we call $M$ a left partial $H$-module. Denote the partial $H$-module structure on $M$ by the map
\begin{equation}
	\label{eq:partialmodule}
\begin{array}{rccl} \bullet : & H\otimes M & \rightarrow  & M \\ 
	\, & h\otimes m & \mapsto & h\bullet m =\pi (h)(m).
\end{array}
\end{equation}
Given two left partial $H$-modules $M$ and $N$, a morphism of partial $H$-modules between them is a linear map $f:M\rightarrow N$, such that $f(h\bullet m)=h\bullet f(m)$ for all $h\in H$ and $m\in M$. The category of left partial $H$-modules is denoted by ${}_H \mathsf{PMod}$.

To the Hopf algebra $H$, one can associate another algebra, denoted by $H_{par}$. This algebra factorizes partial representations of $H$ by algebra morphisms.

\begin{definition}[{\cite{ABVparreps}}]
	Given a Hopf algebra $H$, the \emph{partial Hopf algebra} $H_{par}$ is defined as the quotient $H_{par}= T(H)/\mathcal{I}$, in which $T(H)$ is the tensor algebra over $H$ and $\mathcal{I}$ is the ideal generated by the elements of the form
	\begin{gather*}
		 1_H -1_{T(H)} ;\\  
		 h\otimes k_{(1)} \otimes S(k_{(2)}) -hk_{(1)} \otimes S(k_{(2)}) ;\\
		h_{(1)} \otimes S(h_{(2)}) \otimes k - h_{(1)} \otimes S(h_{(2)}) k 
	\end{gather*}
	for all $h,k\in H$. 
\end{definition}

Denoting by $[h]$ the class of an element $h\in H$ inside this algebra $H_{par}$, it is easy to see that the linear map $[-]: H\rightarrow H_{par}$, sending each element $h\in H$ into its class $[h]\in H_{par}$, is a partial representation of $H$. Moreover, for any partial representation $\pi :H\rightarrow B$, there exists a unique algebra morphism $\widehat{\pi} :H_{par} \rightarrow B$ such that $\pi=\widehat{\pi} \circ [-]$. 
In particular, this implies that the category of left partial $H$-modules, ${}_H \mathsf{PMod}$, is isomorphic to the category of left $H_{par}$-modules, ${}_{H_{par}} \mathsf{Mod}$. The \(H_{par}\)-module structure on a partial module as in \eqref{eq:partialmodule} is given by
\[[h_1] \cdots [h_n] \triangleright m = h_1 \bullet (\cdots \bullet (h_n \bullet m)).\]
	 
There is an important subalgebra of the algebra $H_{par}$, namely the algebra
\begin{equation}
\label{eq:Apar}
A_{par}(H) =\left\langle \varepsilon_h =[h_{(1)}][S(h_{(2)})]\in H_{par} \mid h\in H \right\rangle .
\end{equation}
We will often write just \(A_{par}\) instead of \(A_{par}(H)\).
One can define a partial action of the Hopf algebra $H$ on $A_{par}(H)$ such that the algebra $H_{par}$ is isomorphic to the partial smash product $A_{par} \underline{\#} H$ \cite{ABVparreps}. If the antipode is invertible, then every left partial $H$-module $M$ (viewed as an $H_{par}$-module) can be endowed with the structure of an $A_{par}$-bimodule, given by:
\[
\varepsilon_h \cdot m \cdot  \varepsilon_k =[h_{(1)}][S(h_{(2)})][k_{(2)}][S^{-1}(k_{(1)})] \triangleright m.
\]

The main results about the algebra $H_{par}$ and its subalgebra $A_{par}$ can be summarized in the following theorem.

\begin{theorem}[{\cite{ABVparreps}}]
    Let \(H\) be a Hopf algebra with invertible antipode.
	\begin{enumerate}[(i),leftmargin=0.8cm]
		\item The algebra $H_{par}$ has the structure of a Hopf algebroid over the subalgebra $A_{par}$.
		\item The category of partial $H$-modules is isomorphic to the category of $H_{par}$-modules, which is a closed monoidal category and for which the forgetful functor $U: {}_{H_{par}}\mathsf{Mod} \rightarrow {}_{A_{par}} \mathsf{Mod}_{A_{par}}$ is closed and strongly monoidal.
	\end{enumerate}
\end{theorem}

\begin{example}
	\label{ex:finitegroups}
	The partial representations of finite groups are classified in \cite{DEP}. It is shown there that if \(H = kG\) for a finite group \(G,\) then \(H_{par}\) is isomorphic to a groupoid algebra \(k\Gamma(G)\), where the objects of the groupoid \(\Gamma(G)\) are subsets of \(G\) containing the unit. In particular it was shown that \(H_{par}\) is a finite-dimensional semisimple algebra if \(k\) is of characteristic \(0\) and algebraically closed. 
	
	The structure of the simple partial \(kG\)-modules is discussed in more detail in \cite{DZ}. Let \(X\) be an object of \(\Gamma(G),\) i.\,e.\ a subset of \(G\) containig the unit, and \(K\) its left stabilizer. Then \(X = \bigcup_i K g_i\) for some representatives \(g_i \in G\). This induces a \((kG_{par}, kK)\)-bimodule \(kX^{-1}\). Indeed, \(X^{-1} = \bigcup_i g_i^{-1} K\) is a basis for \(kX^{-1}\) on which \(G\) acts partially on the left (this is simply the restriction of the regular action on \(G\) to \(X^{-1}\)) and \(K\) acts globally on the right. In fact, this bimodule only depends on the connected component in \(\Gamma(G)\) and not on the object \(X\) picked from that connected component.
    Now \cite[Corollary 2.4]{DZ} states that every simple partial \(kG\)-module is of the form
    \begin{equation}
        \label{eq:kXW}
        kX^{-1} \otimes_{kK} W
    \end{equation}
    for some \(X \subseteq G\) containing the unit and some simple \(kK\)-module \(W\).
\end{example}

\subsection{Algebraic partial comodules}

In \cite{ABQVcoreps}, the notion of an algebraic partial $H$-comodule was introduced. The axioms are exactly dual to those of partial modules.

\begin{definition}
    \label{def:parcomod}
	A \emph{right} (\emph{algebraic}) \emph{partial \(H\)-comodule} is a \(k\)-vector space \(M\) endowed with a \(k\)-linear map \(\rho : M \to M \otimes H,\) satisfying
	\begin{enumerate}[(PCM1),leftmargin=1.8cm]
		\item \label{PCM1} $(M\otimes \epsilon) \rho =M$;
		\item  \label{PCM2} $(M\otimes H \otimes \mu)(M\otimes H \otimes H \otimes S)(M\otimes \Delta \otimes H) (\rho \otimes H )\rho =(M\otimes H \otimes \mu)(M\otimes H \otimes H \otimes S)$\\
		$(\rho \otimes H \otimes H)(\rho \otimes H)\rho$;
		\item \label{PCM3} $(M\otimes \mu \otimes H)(M\otimes H \otimes S \otimes H) (M\otimes H \otimes \Delta)(\rho \otimes H)\rho =(M\otimes \mu \otimes H)(M\otimes H \otimes S \otimes H)$\\
		$(\rho \otimes H \otimes H)(\rho \otimes H)\rho$.
	\end{enumerate}
	The first axiom tells that \(\rho\) is counital, while \ref{PCM2} and \ref{PCM3} express the \textit{partial coassociativity}. 
	As shown in \cite[Lemma 3.3]{ABQVcoreps}, axioms \ref{PCM2} and \ref{PCM3} are equivalent to
	\begin{enumerate}[(PCM1),leftmargin=1.8cm]
		\setcounter{enumi}{3}
		\item  \label{PCM4} $(M\otimes H\otimes \mu)(M\otimes H \otimes S \otimes H)(M\otimes \Delta \otimes  H)(\rho \otimes H)\rho=(M\otimes H\otimes \mu)(M\otimes H \otimes S \otimes H)$\\
		$(\rho \otimes H \otimes H)(\rho \otimes H)\rho$;
		\item \label{PCM5} $(M\otimes \mu \otimes H)(M\otimes S\otimes H \otimes H)(M\otimes H \otimes \Delta)(\rho \otimes H)\rho =(M\otimes \mu \otimes H)(M\otimes S\otimes H \otimes H)$\\
		$(\rho \otimes H \otimes H)(\rho \otimes H)\rho$.
	\end{enumerate}
	
	A \emph{morphism between partial comodules} \((M, \rho)\) and \((N, \sigma)\) is a linear map \(f : M \to N\) such that \((f \otimes H) \rho = \sigma f\). 
	The category of right partial comodules over \(H\) is denoted by \(\mathsf{PMod}^H\). 
\end{definition}

The reason why the foregoing partial comodules are called \emph{algebraic} is to distinguish them from the \emph{geometric} partial comodules introduced in \cite{JoostJiawey} and further studied in \cite{PJ1,PJ2,PJ3}. 
Although any algebraic partial comodule can be endowed with a geometric partial comodule structure, as highlighted in \cite[Proposition 3.20]{PJ3}, the converse is not necessarily true (see e.\,g.\ \cite[Example 3.14 and the preceeding paragraph]{ABQVcoreps}).

Using Sweedler notation, we write \(\rho(m) = m^{(0)} \otimes m^{(1)}\) and the axioms of partial comodules read, for all \(m \in M\),
\begin{enumerate}[(PCM1),leftmargin=1.8cm]
	\item \(m^{(0)} \epsilon(m^{(1)}) = m;\)
	\item \(m^{(0)(0)} \otimes {m^{(0)(1)}}_{(1)} \otimes {m^{(0)(1)}}_{(2)} S(m^{(1)}) = m^{(0)(0)(0)} \otimes m^{(0)(0)(1)} \otimes m^{(0)(1)} S(m^{(1)});\)
	\item \(m^{(0)(0)} \otimes {m^{(0)(1)}} S({m^{(1)}}_{(1)}) \otimes {m^{(1)}}_{(2)} = m^{(0)(0)(0)} \otimes m^{(0)(0)(1)} S(m^{(0)(1)}) \otimes m^{(1)};\)
	\item \(m^{(0)(0)} \otimes {m^{(0)(1)}}_{(1)} \otimes S({m^{(0)(1)}}_{(2)}) m^{(1)} = m^{(0)(0)(0)} \otimes m^{(0)(0)(1)} \otimes S(m^{(0)(1)}) m^{(1)};\)
	\item \(m^{(0)(0)} \otimes S({m^{(0)(1)}}) {m^{(1)}}_{(1)} \otimes {m^{(1)}}_{(2)} = m^{(0)(0)(0)} \otimes S(m^{(0)(0)(1)}) m^{(0)(1)} \otimes m^{(1)}.\)
\end{enumerate}

As might be expected, there is a strong connection between partial \(H\)-comodules and partial modules over the (finite) dual of \(H\).

\begin{theorem}[\cite{ABQVcoreps}]
    \label{thm:isomodcomod}
	Let \(H\) be a Hopf algebra over a field \(k\) such that \(H^\circ\) is dense in \(H^*\). Let \(\rho : M \to M \otimes H\) be a linear map. Then \((M, \rho)\) is a partial \(H\)-comodule if and only if \((M, \lambda)\) is a partial \(H^\circ\)-module, where
	\[\lambda : H^\circ \otimes M \to M : h^* \otimes m \mapsto m^{(0)} h^*(m^{(1)}).\]
	If \(H\) is finite-dimensional, then every partial \(H^*\)-module is of this form.
\end{theorem}

In particular, if \(H\) is finite-dimensional,
\begin{equation}
    \label{eq:isomodcomod}
    \PMod^H \simeq {_{H^*}} \PMod \simeq {_{(H^*)_{par}}} \Mod.
\end{equation}

\begin{example}
	For a finite group \(G\), consider the dual group algebra \(k G^*\). Since \(\PMod^{k G^*} \simeq {_{k G}}\PMod\), partial \(kG^*\)-comodules form a semisimple category if \(k\) is algebraically closed and of characteristic \(0\), and their structure is described by the work in \cite{DEP, DZ}. Cf. \cref{ex:finitegroups}.
\end{example}

\begin{example}\label{ex:gtrans}
    Let \(H\) be a Hopf algebra, \((N,\rho_N)\) be a right partial \(H\)-comodule and \(g \in H\) be a grouplike element. Then \(N\) with
    \[\rho_N^g\colon N \to N \otimes H : n \mapsto n^{(0)} \otimes gn^{(1)}\]
    is a right partial \(H\)-comodule.
\end{example}

\begin{example}
    \label{ex:partial_integral}
	As a particular case of interest of \cref{ex:gtrans}, let \(G\) be a finite group and let \(k\) be a field such that \(\mathrm{char}(k)\) does not divide \(|G|\). Take a subgroup \(K\) of \(G\). Then
	\begin{equation}
        \label{eq:partial_integral}
        k \to k \otimes kG : 1 \mapsto 1 \otimes \frac{1}{|K|}\sum_{h \in K} h
    \end{equation}
	defines a partial \(kG\)-comodule structure on \(k\). If \(g \in G \setminus K\) then also
	\begin{equation}
        \label{eq:partial_integral_trans}
	    k \to k \otimes kG : 1 \mapsto 1 \otimes \frac{1}{|K|}\sum_{h \in K} gh
	\end{equation}
	defines a partial \(kG\)-comodule structure on \(k\). In \cref{se:groups} we will show that every \(1\)-dimensional partial \(kG\)-comodule is of this form. 
\end{example}

Partial comodule algebras appeared in \cite{ABenvelope, BVdual, CJ}. Let us recall their definition.

\begin{definition}
	A (\emph{symmetric}) \emph{right partial coaction} of \(H\) on a unital algebra \(B\) is a linear map
	\[\rho : B \to B \otimes H : a \mapsto a^{(0)} \otimes a^{(1)}\]
	such that for all \(a, b \in B\)
	\begin{enumerate}[(PC1)]
		\item \label{PC1} \((B \otimes \epsilon)\rho(a) = a;\)
		\item \label{PC2} \(\rho(ab) = \rho(a) \rho(b);\)
		\item \label{PC3} \((\rho(1_B) \otimes 1_H)\big((B \otimes \Delta)\rho(a)\big) = (\rho \otimes H)\rho(a) = \big((B \otimes \Delta)\rho(a)\big)(\rho(1_B) \otimes 1_H).\)
	\end{enumerate}
	We call \(B\) a \emph{partial comodule algebra}. 
\end{definition}

The motivation for calling such a partial coaction \emph{symmetric} is to distinguish it from the one introduced in \cite{CJ}, where only one of the equalities in \ref{PC3} was imposed. However, here we are only concerned with symmetric partial coactions and hence we will often omit to highlight it.

It is a direct check that a partial comodule algebra is a partial comodule in the sense of \cref{def:parcomod}.

\subsection{Right coideal subalgebras}

One of the key ingredients of our construction of partial comodules are right coideal subalgebras.

\begin{definition}
	A right coideal subalgebra of \(H\) is a subalgebra \(A\) such that \(\Delta(A) \subseteq A \otimes H\).
\end{definition}

We also have the following correspondence, originally due to Takeuchi (see \cite[Proposition 1]{Takeuchi}).
Our rephrasing is slightly richer than the original. The interested reader may refer to \cite[\S3.1.5]{EGSV} and \cite[\S2]{Saracco} for an analogue of this result in the more general bialgebroid setting.

\begin{proposition}
	\label{le:correspondance_idealcoideal}
	Let \(H\) be a Hopf algebra. Then we have a monotone Galois connection (or, equivalently, an adjunction)  between the following posets:
	\begin{equation}
		\label{eq:correspondance_coidealsubalg}
        \xymatrix @R=0pt{
		\left\{\text{right coideal subalgebras of $H$}\right\} \ar@<+0.5ex>[r]^-{\Phi} & \left\{\text{left ideal coideals of $H$}\right\} \ar@<+0.5ex>[l]^-{\Psi} \\
		A \ar@{|->}[r] & H A^+  \\
		{^{\mathrm{co} H/B}} H & B \ar@{|->}[l]
        }
	\end{equation}
	where \(A^+ =  A \cap \ker \epsilon\) and \({^{\mathrm{co} H/B}} H = \{h \in H \mid \pi(h_{(1)}) \otimes h_{(2)} = \pi(1_H) \otimes h\}\), \(\pi : H \to H/B\) being the canonical projection.
\end{proposition}

So for any right coideal subalgebra \(A\) one can construct the coalgebra and left \(H\)-module quotient
\begin{equation}
	\label{eq:Hbar}
	\pi : H \to H/HA^+ =: \bar{H}.
\end{equation}

\begin{example}
\label{ex:rcs}
	\begin{enumerate}[(i),leftmargin=0.8cm]
		\item Obviously, any Hopf subalgebra is a right coideal subalgebra. If \(H = kG\) is the group algebra of a finite group, then \(A\) is a right coideal subalgebra precisely when it is a Hopf subalgebra.
		\item Let \(G\) be a finite group and consider the dual group algebra \(kG^*\), with basis \(\{p_g \mid g \in G\}\) which is dual to the basis \(G\) of \(kG\). For a subgroup \(K\) of \(G\), 
		\[\left \langle \sum_{h \in K} p_{hg} \mid g \in G \right \rangle\]
		is a right coideal subalgebra of \(kG^*\) which is a Hopf subalgebra if and only if \(K\) is a normal subgroup of \(G\). In fact, every right coideal subalgebra \(A\) is of this form (see e.\,g.\ \cite[Example 5.19]{Natale}). Indeed, since \(A\) is a commutative semisimple algebra (see e.\,g.\ \cite[Lemma 4.0.2]{Bkernels}), it has a basis of primitive idempotents. Let \(e=\sum_{g \in X} p_g\) the primitive idempotent for which \(1 \in X\). Then, since \(A\) is a right \(kG^*\)-comodule, \(\sum_{g \in Xh} p_g \in A\) for all \(h \in G\). Moreover, \(Xh \cap Xh' = \varnothing\) or \(Xh = Xh'\) for any \(h, h' \in G\) because if not \(e\) would not be primitive. Now it is easy to see that \(X\) is a subgroup of \(G\).	\qedhere
	\end{enumerate}
\end{example}

\section{Construction of partial comodules}
\label{se:construction}

\subsection{The general construction}

The following variation of \cite[Proposition 8]{ABenvelope} will be useful in what follows.
\begin{proposition}
	\label{prop:partialcoaction_ideal}
	Let \(A\) be a right \(H\)-comodule algebra with coaction \(\rho : A \to A \otimes H\). If \(I\) is an ideal of \(A\) possessing a unit \(e\) (that is, $xe=x=ex$ for all $x\in I$), then the map
	\begin{equation}\label{eq:rhoI}
 \rho_I(x) \coloneqq \rho(x)(e \otimes 1_H) = x^{(0)} e \otimes x^{(1)}
 \end{equation}
	defines a symmetric 
    right partial coaction of \(H\) on \(I\). 
\end{proposition}
\begin{proof}
Remark that \(e\) is central in \(A,\) since for any \(a \in A,\) both \(ea\) and \(ae\) are in \(I\) and equal to \(eae\) because \(e\) is a unit for \(I\).
    For \(x \in I,\)
    \[x^{(0)} e \epsilon( x^{(1)}) = xe = x\]
    so \ref{PC1} is satisfied. If \(x, y \in I,\) then
    \[\rho_I(x) \rho_I(y) = x^{(0)} e y^{(0)} e \otimes x^{(1)} y^{(1)} = x^{(0)} y^{(0)} e^2 \otimes x^{(1)} y^{(1)} = (xy)^{(0)} e \otimes (xy)^{(1)}\]
    because \(A\) is a right \(H\)-comodule algebra, so also \ref{PC2} holds. Finally for \ref{PC3}, 
    \begin{align*}
        (\rho_I \otimes H)\rho_I(x) = \rho_I(x^{(0)} e) \otimes x^{(1)} &= x^{(0)} e^{(0)} e \otimes x^{(1)} e^{(1)} \otimes x^{(2)} \\& = (x^{(0)} e \otimes {x^{(1)}}_{(1)} \otimes {x^{(1)}}_{(2)})(e^{(0)} \otimes e^{(1)} \otimes 1_H), \\
        (\rho_I \otimes H)\rho_I(x) = \rho_I(e x^{(0)} ) \otimes x^{(1)} &= e^{(0)} x^{(0)}  e \otimes e^{(1)} x^{(1)}  \otimes x^{(2)} \\& = (e^{(0)} \otimes e^{(1)} \otimes 1_H) (x^{(0)} e \otimes {x^{(1)}}_{(1)} \otimes {x^{(1)}}_{(2)}). \qedhere
    \end{align*}
\end{proof}

We will now restrict ourselves to the case where the right \(H\)-comodule algebra \(A\) is contained in \(H\) (i.\,e.\ \(A\) is a right coideal subalgebras) and where the idempotent $e$ is central in \(A\). It turns out that the idempotents of \(H\) obtained this way are exactly those that satisfy \(ee_{(1)} \otimes e_{(2)} = e_{(1)}e \otimes e_{(2)}\).

\begin{lemma}
	\label{le:ecentral}
	Let \(e\) be idempotent in \(H\) and let \(A\) be the smallest right coideal subalgebra of \(H\) containing \(e\). Then \(e\) is central in \(A\) if and only if \(ee_{(1)} \otimes e_{(2)} = e_{(1)}e \otimes e_{(2)}\).
\end{lemma}
\begin{proof}
	Suppose that \(ee_{(1)} \otimes e_{(2)} = e_{(1)}e \otimes e_{(2)}\). Write
	\[
	e_{(1)} \otimes e_{(2)} = \sum_{i \in E} a_i \otimes b_i
	\]
	with \(\{b_i \mid i \in E\}\) linearly independent.
	Let \(A\) be the subalgebra generated by the elements \(a_i\). Then \(A\) is a right coideal subalgebra: we know that
	\begin{equation}\label{deltaAi}
		\sum_{i} \Delta(a_i) \otimes b_i = \sum_{i} a_i \otimes \Delta(b_i) ,
	\end{equation}
	and since the \(b_i\) are linearly independent, there exist linear functionals \( b_i^* \in H^*\) such that \( b_i^* (b_j)=\delta_{i,j}\). By applying \(   H \otimes H \otimes b_i^* \) to the equality (\ref{deltaAi}), we conclude that   
	\(\Delta(a_i) \in A \otimes H\) for all \(i\). By the multiplicativity of $\Delta$, we conclude that $\Delta(A) \subseteq A \otimes H$. Since \(\sum_{i} ea_i \otimes b_i = \sum_{i} a_ie \otimes b_i\), \(e\) commutes with all \(a_i\) (again by linear independence of the \(b_i\); we apply \(H \otimes b_i^*\) to the previous equality), and it follows that \(e\) is central in \(A\).
	
	Conversely, if \(e\) is central in \(A\) and \(\Delta(A) \subseteq A \otimes H,\) then clearly \(ee_{(1)} \otimes e_{(2)} = e_{(1)}e \otimes e_{(2)}\).
\end{proof}

The above motivates the following definition.

\begin{definition}
    A \textit{subcentral} idempotent of \(H\) is a non-zero idempotent \(e \in H\) such that 
    \begin{equation}\label{eq:subcentr}
    ee_{(1)} \otimes e_{(2)} = e_{(1)} e \otimes e_{(2)}.
    \end{equation}
\end{definition}

From now on, for any subcentral idempotent \(e \in H\), let \(A_e\) be the right coideal subalgebra generated by \(e\). When there is no risk of confusion, we will simply denote it by \(A\) instead of \(A_e\). We denote by \(\bar{H}_e\) (or simply \(\bar H\)) the coalgebra and left \(H\)-module \(H/HA_e^+\) and by \(\pi_e\) (or simply \(\pi\)) the coalgebra and left \(H\)-linear projection \(H \to \bar{H}\), as we did in \eqref{eq:Hbar}.
Remark that by \cref{prop:partialcoaction_ideal} the ideal \(I\) in \(A\) generated by \(e\) is a partial comodule algebra by means of
\[\rho : I \to I \otimes H : x \mapsto x_{(1)}e \otimes x_{(2)}.\]
However, this does certainly not cover all partial \(H\)-comodules: for instance in \cref{ex:partial_integral}, the partial comodule \eqref{eq:partial_integral} is of the above form, but its translates \eqref{eq:partial_integral_trans} are not necessarily. We now provide a more general construction that allows us to deal with these additional cases, starting with a technical lemma.

\begin{lemma}
	\label{le:Hbarcomodule}
	Let \(e \in H\) be a subcentral idempotent. Then the left ideal \(He\) is a left (global) \(\bar{H}_e\)- and a right partial \(H\)-bicomodule, where the left \(\bar{H}_e\)-comodule structure is given by the restriction of \((\pi_e \otimes H) \circ \Delta\) and the right partial \(H\)-comodule structure is given by
    \begin{equation}\label{eq:rhoe}
    \rho_e \colon He \to He \otimes H : x \mapsto x_{(1)}e \otimes x_{(2)}.
    \end{equation}
\end{lemma}
\begin{proof}
	 Since for any \(a \in A\) we have that \(a - \epsilon(a)1_H \in A^+\), it follows that \(\pi(ha) = \pi(h\epsilon(a))\) for all \(h \in H\). Therefore, since \(e_{(1)} \otimes e_{(2)} \in A \otimes H,\) we have that for any \(h \in H\)
	\begin{equation}\label{eq:Hbarcoaction} 
        \begin{aligned}
        (\pi \otimes H) \Delta(he)  &= \pi(h_{(1)} e_{(1)}) \otimes h_{(2)} e_{(2)} \\ 
        &= \pi(h_{(1)} \epsilon(e_{(1)})) \otimes h_{(2)} e_{(2)}  \\ 
        &= \pi(h_{(1)}) \otimes  h_{(2)}e \in \bar{H} \otimes He 
	\end{aligned}
        \end{equation}
        and so \(He\) is a left \(\bar H\)-subcomodule of \(H\).
        
        Concerning the fact that it is also a right partial \(H\)-comodule, the counitality follows because for \(x \in He\), we have \(xe = x\). Now, if we apply $\Delta \otimes H$ on \(ee_{(1)} \otimes e_{(2)} = e_{(1)}e \otimes e_{(2)}\), we get
	\begin{equation}\label{e1e1'}
		e_{(1')} e_{(1)} \otimes e_{(2')} e_{(2)} \otimes e_{(3)} = e_{(1)} e_{(1')} \otimes e_{(2)} e_{(2')} \otimes e_{(3)}.
	\end{equation}
	Thus
	\begin{align*}
		\rho^2(x) &\stackrel{\phantom{\eqref{e1e1'}}}{=} x_{(1)} e_{(1)}e \otimes x_{(2)} e_{(2)} \otimes x_{(3)}, \\
		\rho^3(x)
		&\stackrel{\phantom{\eqref{e1e1'}}}{=} x_{(1)} e_{(1)} e_{(1')} e \otimes x_{(2 )}e_{(2)} e_{(2')}   \otimes x_{(3)} e_{(3)} \otimes x_{(4)} \\
		&\stackrel{\eqref{e1e1'}}{=} x_{(1)} e_{(1')} e_{(1)}  e \otimes x_{(2)} e_{(2')} e_{(2)}    \otimes x_{(3)}e_{(3)} \otimes x_{(4)},
	\end{align*}
	from where we will now deduce the partial coassociativity axioms.
 \begin{description}[leftmargin=0pt]
\item[\ref{PCM2}] On the one hand, we have
	\[ 
	(He \otimes H \otimes \mu)(He\otimes H \otimes H \otimes S)(He \otimes \Delta \otimes H) \rho^2 (x)   =  x_{(1)} e_{(1)} e \otimes x_{(2)} e_{(2)} \otimes x_{(3)} e_{(3)} S(x_{(4)}) .
	\]
	On the other hand,
	\begin{align*}
	 (He \otimes H \otimes \mu) & (He \otimes H \otimes H \otimes S) \rho^3 (x) = x_{(1)} e_{(1)} e_{(1')} e \otimes x_{(2)} e_{(2)} e_{(2')} \otimes x_{(3)} e_{(3)} S(x_{(4)}) \\
		& \stackrel{\phantom{\eqref{e1e1'}}}{=}  x_{(1)} e_{(1)} e_{(1')} e \otimes x_{(2)} e_{(2)} e_{(2')} \otimes x_{(3)} e_{(3)} e_{(3')} S(e_{(4')}) S(x_{(4)}) \\
		& \stackrel{\eqref{e1e1'}}{=}  x_{(1)} e_{(1')} e_{(1)} e \otimes x_{(2)} e_{(2')} e_{(2)} \otimes x_{(3)} e_{(3')} e_{(3)} S(x_{(4)}e_{(4')})  \\
		& \stackrel{\phantom{\eqref{e1e1'}}}{=}  x_{(1)} e_{(1)} e \otimes x_{(2)} e_{(2)} \otimes x_{(3)} e_{(3)} S(x_{(4)}) ,
	\end{align*}
	where the last equality follows from the fact that $x=xe$.
	
\item[\ref{PCM3}] We have
	\[
	(He \otimes \mu \otimes H)(He \otimes H \otimes S \otimes H) (He \otimes H \otimes \Delta )\rho^2 (x) =x_{(1)} e_{(1)}e \otimes x_{(2)} e_{(2)} S(x_{(3)}) \otimes x_{(4)} ,
	\]
	and
	\begin{align*}
		 (He \otimes \mu \otimes H)(He \otimes H \otimes S \otimes H) \rho^3 (x) & \stackrel{\phantom{\eqref{e1e1'}}}{=} x_{(1)} e_{(1)} e_{(1')} e \otimes x_{(2)} e_{(2)} e_{(2')} S(x_{(3)} e_{(3)}) \otimes x_{(4)} \\
		& \stackrel{\eqref{e1e1'}}{=}  x_{(1)} e_{(1')} e_{(1)} e \otimes x_{(2)} e_{(2')} e_{(2)} S(e_{(3)})S( x_{(3)}) \otimes x_{(4)} \\
		& \stackrel{\phantom{\eqref{e1e1'}}}{=}  x_{(1)} e_{(1')} e e \otimes x_{(2)} e_{(2')} S( x_{(3)}) \otimes x_{(4)} \\
		& \stackrel{\phantom{\eqref{e1e1'}}}{=}  x_{(1)} e_{(1)}e \otimes x_{(2)} e_{(2)} S(x_{(3)}) \otimes x_{(4)} ,
	\end{align*}
	where the last equality is true because $e$ is an idempotent.

    Finally, we show that the left \(\bar{H}_e\)-coaction and the right partial \(H\)-coaction commute. Indeed,
    \begin{align*}
        (\pi \otimes H \otimes H)  (\Delta \otimes H) \rho_e(he) &=
        \pi(h_{(1)}e_{(1)} e_{(1')}) \otimes h_{(2)} e_{(2)} e_{(2')} \otimes h_{(3)} e_{(3)} \\
        &= \pi(h_{(1)}) \otimes h_{(2)} e_{(1)} e \otimes h_{(3)} e_{(2)} \\
        &= (H \otimes \rho_e)(\pi \otimes H) \Delta (he). \qedhere
    \end{align*}
 \end{description}
\end{proof}

The following is now immediate.

\begin{corollary}
	\label{le:cotensorcomodule}
	Let \(e \in H\) be a subcentral idempotent and let \((W, \rho^W)\) be a right \(\bar{H}_e\)-comodule. Then the cotensor product \(W \cotensor^{\bar{H}_e} He\) is a right partial \(H\)-comodule by
	\[\rho : W \cotensor^{\bar{H}_e} He \to (W \cotensor^{\bar{H}_e} He) \otimes H : \sum_i w_i \otimes x_i \mapsto \sum_i w_i \otimes {x_i}_{(1)}e \otimes {x_i}_{(2)}.\]
\end{corollary}

\begin{construction}
	\label{construction:general2}
	We can summarize our construction of partial comodules as follows. 
	
	\begin{itemize}[leftmargin=0.8cm]
		\item Take an idempotent \(e \in H\) for which \(ee_{(1)} \otimes e_{(2)} = e_{(1)}e \otimes e_{(2)}\).
		\item Let \(A_e\) be the right coideal subalgebra generated by \(e\). Then \(e\) is central in \(A_e\) by \cref{le:ecentral}.
		\item Consider \(A_e^+ = A_e \cap \ker (\epsilon)\) and \(HA_e^+\). The latter is a left ideal and a (twosided) coideal of \(H\) by \cref{le:correspondance_idealcoideal}. As a consequence, \(\bar{H}_e = H/HA_e^+\) is a coalgebra and a left \(H\)-module. 
		\item The left ideal \(He\) in \(H\) is also a left \(\bar{H}_e\)-comodule by \cref{le:Hbarcomodule}.
		\item Take a right \(\bar{H}_e\)-comodule \(W\). By \cref{le:cotensorcomodule}, the cotensor product \(W \cotensor^{\bar{H}_e} He \)
		is a right partial \(H\)-comodule with partial coaction
		\(w \otimes x \mapsto w \otimes x_{(1)}e \otimes x_{(2)}\) (summation understood).
	\end{itemize}
    By \cref{le:ecentral}, the first two steps can be replaced by taking a right coideal subalgebra \(A\) of \(H\) and an idempotent \(e\)  which is central in \(A\) and which generates \(A\) as a right coideal subalgebra.
	\end{construction}

Let us make a note on partial comodule algebras.
Since \(\pi_e \colon H \to \bar{H}_e\) is a coalgebra morphism, \(\pi_e(1_H)\) is grouplike in \(\bar{H}_e\). This induces a ``trivial'' \(1\)-dimensional right \(\bar{H}_e\)-comodule \(T\). 

\begin{lemma}\label{le:partialcomodulealgebra} 
    Let \(e \in H\) be a subcentral idempotent.
    Taking for \(W\) the trivial \(\bar H_e\)-comodule \(T\) in the last step of \cref{construction:general2} gives
    \[T \cotensor^{\bar{H}_e} He \cong \{he \in He \mid \pi_e(1_H) \otimes he = \pi_e(h_{(1)}) \otimes h_{(2)}e\} = \prescript{\mathrm{co}\bar{H}_e}{}{H}\, e,\]
    which is a right partial comodule with respect to \eqref{eq:rhoI}:
    \[\rho \colon \prescript{\mathrm{co}\bar{H}_e}{}{H}\, e \to \prescript{\mathrm{co}\bar{H}_e}{}{H}\, e \otimes H : he \mapsto h_{(1)}e_{(1)}e \otimes h_{(2)}e_{(2)}.\]
\end{lemma}

\begin{proof}
Recall from \eqref{eq:Hbarcoaction} that the left \(\bar{H}\)-comodule structure on \(He\) is given by \(he \mapsto \pi(h_{(1)}) \otimes h_{(2)} e\), so that indeed 
 \[T \cotensor^{\bar{H}} He \cong \{he \in He \mid \pi(1_H) \otimes he = \pi(h_{(1)}) \otimes h_{(2)}e\} = \prescript{\mathrm{co}\bar{H}}{}{(He)}.\]
 Clearly, \(\prescript{\mathrm{co}\bar{H}}{}{H}\, e \subseteq \prescript{\mathrm{co}\bar{H}}{}{(He)} \).  Conversely, if \(he \in \prescript{\mathrm{co}\bar{H}}{}{(He)}\), then 
\[\pi(1_H) \otimes he = \pi(h_{(1)}) \otimes h_{(2)} e = \pi(h_{(1)}\epsilon(e_{(1)})) \otimes h_{(2)} e_{(2)} = \pi((he)_{(1)}) \otimes (he)_{(2)}.\]
Thus \(he \in \prescript{\mathrm{co}\bar{H}}{}{H}\) and so \(he = he^2 \in \prescript{\mathrm{co}\bar{H}}{}{H}\,e\).
\end{proof}

\begin{remark}
\label{rk:comodulealgebra}
Let \(e \in H\) be a subcentral idempotent and let \(A\) be the right coideal subalgebra of \(H\) generated by \(e\). Suppose that there exists an \(H\)-module which is faithfully flat as left \(A\)-module. By \cite[Theorem 1]{Takeuchi}, \( \prescript{\mathrm{co}\bar{H}}{}{H} = A\). Thus, applying \cref{le:partialcomodulealgebra} we obtain a right partial comodule
    \[T  \cotensor^{\bar{H}_e} He \cong \prescript{\mathrm{co}\bar{H}_e}{}{H}\,e =  Ae,\]
    which coincides with the partial comodule algebra from \cref{prop:partialcoaction_ideal}.
This hypothesis is satisfied in particular when \(H\) is finite-dimensional, in view of, e.\,g.\ \cite[Theorem 6.1]{Sprojectivity}.
Since we will be mainly concerned with finite-dimensional Hopf algebras, this condition will always be satisfied in the following sections.
\end{remark}

\begin{proposition}\label{prop:Ierhog}
    Let \(e \in H\) be a subcentral idempotent. Let \(I_e \coloneqq \prescript{\mathrm{co}\bar{H}_e}{}{H}\,e\) be the corresponding right partial comodule from \cref{le:partialcomodulealgebra}. Then, for every grouplike element \(g \in H\) we can construct the \(1\)-dimensional right \(\bar H_e\)-comodule \(W_g = k\) with \(1_k \mapsto 1_k \otimes \pi_e(g)\) and the right partial \(H\)-comodule \(I_e\) with \(\rho^g\colon x \mapsto x_{(1)}e \otimes gx_{(2)}\) as in \cref{ex:gtrans}. We have that
    \[W_g   \cotensor^{\bar{H}_e} He \cong (I_e,\rho^g)\]
    as partial comodules.
\end{proposition}

\begin{proof}
    By the usual argument,
    \[W_g  \cotensor^{\bar{H}} He \cong \{he \in He \mid \pi(g) \otimes he = \pi(h_{(1)}) \otimes h_{(2)}e\}.\]
    If \(x \in I_e = \prescript{\mathrm{co}\bar{H}}{}{(He)}\) (see the proof of \cref{le:partialcomodulealgebra}), then the left \(H\)-linearity of \(\pi\) allows us to verify that \(gx\) satisfies
    \[\pi(gx_{(1)}) \otimes gx_{(2)} = \pi(g) \otimes gx\]
    and so \(gx \in W_g  \cotensor^{\bar{H}} He\). Conversely, if \(he \in W_g  \cotensor^{\bar{H}} He\), then \(g^{-1}he\) satisfies
    \[\pi\big(g^{-1}h_{(1)}\big) \otimes g^{-1}h_{(2)}e = \pi(1) \otimes g^{-1}he,\]
    again by left \(H\)-linearity of \(\pi\). It follows that
    \[
    \xymatrix@R=0pt{
    \big(I_e,\rho^g\big) \ar@{<->}[r] & W_g   \cotensor^{\bar{H}_e} He \\
    he \ar@{|->}[r] & ghe \\
    g^{-1}h'e & h'e \ar@{|->}[l]
    }
    \]
    is an isomorphism of partial comodules.
\end{proof}

In \cref{se:groups}, we will discuss in more detail the structure of the partial comodules produced by \cref{construction:general2} if \(H\) is the group algebra of a finite group. It will be determined under what conditions \(W \cotensor^{\bar{H}_e} He\) is a simple partial comodule and when two of them are isomorphic. However, first we consider \(H = kG^*\) for a finite group \(G\). 

\subsection{Partial comodules of \texorpdfstring{\(kG^*\)}{kG*}}

We will show
that if \(H = kG^*\) for a finite group \(G\), then \cref{construction:general2} recovers the construction of partial modules of finite groups from \cite{DZ} as explained in \cref{ex:finitegroups}. Indeed, partial comodules over \(kG^*\) are equivalent to partial modules over \(kG\), and using the results from \cite{DZ} we will conclude that every simple partial comodule of the dual  group algebra of a finite group \(G\) is obtained by our construction. 

Recall from \cref{ex:rcs} that every right coideal subalgebra of \(kG^*\) is of the form
\[A = \left\langle \sum_{h \in K} p_{hg} \mid g \in G \right\rangle = \langle p_{Kg_i} \mid i = 1, \dots, n \rangle\]
where \(K\) is a subgroup of \(G\) and \(\{g_1, \dots, g_n\}\) is a set of (randomly chosen) representatives of the right cosets of \(K\) in \(G\), so that \(n = [G : K]\) and \(G = \bigcup_{i = 1}^n Kg_i\). We denoted \(p_{Kg_i} = \sum_{h \in K} p_{hg_i}\), and take \(g_1 = 1_G\).

Since \(kG^*\) is commutative, any idempotent of \(A\) is central in \(kG^*\). This means that the subcentral idempotents are exactly all idempotents of \(kG^*,\) which are sums of \(p_g\)'s. 
An idempotent corresponds thus to a subset \(S\) of \(G\). Let \(K\) be the left stabiliser of \(S\), i.\,e.\ the largest subgroup of \(G\) such that \(KS = S\). Then \(S = \bigcup_{i \in F} Kg_i\) for some \(F \subseteq \{1, \dots, n\}\), and \(e = \sum_{i \in F} p_{Kg_i}\). The right coideal subalgebra 
\[A = \langle p_{Kg_i} \mid i = 1,\dots, n\rangle\]
corresponding to the subgroup \(K\) is the smallest that contains \(e\): if \(A'\) is a right coideal subalgebra such that \(e \in A' \subseteq A,\) then it must come from a subgroup \(K' \supseteq K\). But then \(e = \sum_{j \in F'} p_{K'g'_j}\) and \(S = \bigcup_{j \in F'} K'g'_j\) for some representatives \(g'_j\) and finite set \(F'.\) So \(K'\) also stablises \(S\) to the left. It follows that \(K' = K\), and that \(A\) is indeed the right coideal subalgebra generated by \(e\).

Now
\begin{align*}
	Ae &= \langle p_{Kg_i} \mid i \in F \rangle, \\
	(kG^*)e &= \langle p_{hg_i} \mid h \in K, i \in F \rangle.
\end{align*}

We also have
\begin{align*}
	A^+ &= \langle p_{Kg_i} \mid i = 2, \dots, n \rangle \\
	(kG^*) A^+  &= \langle p_{hg_i} \mid h \in K, i = 2, \dots, n \rangle \\
	\overline{kG^*} &= (kG^*)/(kG^*)A^+  \cong kK^* = \langle p_h \mid h \in K \rangle.
\end{align*}
The left \(kK^*\)-comodule structure on \((kG^*) e\) is given by (see \cref{le:Hbarcomodule})
\[p_{hg_i} \mapsto \sum_{h' \in K} p_{h'} \otimes p_{h'^{-1} h g_i} \in kK^* \otimes (kG^*) e.\]
Remark that a simple right \(kK^*\)-comodule is nothing else than an irreducible \(K\)-representation. 

Let \(M\) be a simple partial \(kG^*\)-comodule, i.\,e.\ a simple partial \(kG\)-module. As explained in \cref{ex:finitegroups}, \(M\) is of the form \(kX^{-1} \otimes_{kK} W\) for some \(X\subseteq G\) containig the unit and some simple left \(kK\)-module \(W\) (where \(K\) is the left stabilizer of \(X\)). This subset \(X\) of \(G\) corresponds to a central idempotent
\[e = \sum_{g \in X} p_g = \sum_{i \in F} p_{Kg_i}\]
where \(X = \bigcup_{i \in F} K g_i\). Applying \cref{construction:general2} with this \(e\) and \(W\) we get a partial \(kG^*\)-comodule \(W  \cotensor^{kK^*} (kG^*) e\). 

\begin{lemma}
    \label{le:cotensor_form}
    Using the notation from above,
	\[W \cotensor^{kK^*} (kG^*) e = \left\langle w^{(0)} \otimes \sum_{h \in K} w^{(1)}(h) p_{hg_i} \mid w \in W, i \in F \right\rangle.\]
\end{lemma}
\begin{proof}
	It is a direct calculation that the elements \(w^{(0)} \otimes \sum_{h \in K} w^{(1)}(h) p_{hg_i}\) belong indeed to \(W \cotensor^{kK^*} (kG^*) e\).
	
	Conversely, suppose that \(x = \sum_j w_j \otimes \sum_{i \in F} \sum_{h \in K} \lambda_{h, i, j} p_{hg_i} \in W \cotensor^{kK^*} (kG^*) e\). Then
	\begin{align*}
		\sum_j w_j^{(0)} \otimes \sum_{h' \in K} w_j^{(1)}(h') p_{h'} \otimes \sum_{h \in K} \sum_{i \in F} \lambda_{h, i, j} p_{hg_i} &= \sum_j w_j \otimes \sum_{i \in F} \sum_{h, h' \in K} \lambda_{h, i, j} p_{h'} \otimes p_{h'^{-1} hg_i} \\
		&= \sum_j w_j \otimes \sum_{i \in F} \sum_{h, h' \in K} \lambda_{h'h, i, j} p_{h'} \otimes p_{hg_i}.
	\end{align*}
In particular, for every \(i \in F,\) 
\[\sum_j w_j^{(0)} \otimes \sum_{h' \in K} \lambda_{1, i, j} w_j^{(1)}(h') p_{h'} = \sum_j w_j \otimes \sum_{h' \in K} \lambda_{h', i, j} p_{h'}.\]
It follows that also
\[\sum_j w_j^{(0)} \otimes \sum_{h' \in K} \lambda_{1, i, j} w_j^{(1)}(h') p_{h'g_i} = \sum_j w_j \otimes \sum_{h' \in K} \lambda_{h', i, j} p_{h'g_i}\]
and hence
\[x = \sum_j \sum_{i \in F} \lambda_{1, i, j} \left(w_j^{(0)} \otimes \sum_{h' \in K} w_j^{(1)}(h') p_{h'g_i}\right). \qedhere\]
\end{proof}

Consider the linear map
\begin{equation}
    \label{eq:mapcomodulekG*}
    kX^{-1} \otimes W \to W \otimes (kG^*) e : g_i^{-1} h^{-1} \otimes w \mapsto w^{(0)} \otimes \sum_{h' \in K} w^{(1)}(h'h^{-1}) p_{h'g_i}.
\end{equation}
Since
\[w^{(0)} \otimes \sum_{h' \in K} w^{(1)}(h'h^{-1}) p_{h'g_i} = (h^{-1} \cdot w)^{(0)} \otimes \sum_{h' \in K} (h^{-1} \cdot w)^{(1)}(h') p_{h'g_i},\]
its image lies in \(W \cotensor^{kK^*} (kG^*) e\) by \cref{le:cotensor_form}. It also shows that \eqref{eq:mapcomodulekG*} is balanced over \(kK,\) so we get a linear map
\[\theta : kX^{-1} \otimes_{kK} W \to W \cotensor^{kK^*} (kG^*) e.\]
It is surjective because of \cref{le:cotensor_form}, hence bijective because the dimension of both spaces is \([X : K] \dim W\).

Since \(W \cotensor^{kK^*} (kG^*) e\) is a partial \(kG^*\)-comodule, it is a partial \(kG\)-module via the isomorphism of categories resulting from \cref{thm:isomodcomod}.
Take \(g \in G\). Let \(j \in \{1,\dots, n\}\) and \(h' \in K\) be such that \(g_i g^{-1} = h' g_j\). The partial \(kG\)-module structure is given by
\begin{align*}
    [g] \cdot \left(w^{(0)} \otimes \sum_{h \in K} w^{(1)}(h) p_{hg_i} \right) &= w^{(0)} \otimes \sum_{h \in K} w^{(1)}(h) p_{hg_ig^{-1}} \\
    &= w^{(0)} \otimes \sum_{h \in K} w^{(1)}(h) p_{hh'g_j}
\end{align*}
if \(j \in F\).  If \(j \notin F,\) then \( [g] \cdot \left(w^{(0)} \otimes \sum_{h \in K} w^{(1)}(h) p_{hg_i} \right) = 0\). From this it is easy to check that \(\theta\) is in fact an isomorphism of partial \(kG\)-modules.
This shows that \cref{construction:general2} produces all simple partial \(kG^*\)-comodules. 
\begin{theorem}
    \label{th:kG*}
    Let \(G\) a finite group. Then every simple partial \(kG^*\)-comodule is of the form \(W \cotensor^{\overline{kG^*}} (kG^*)e\) for a central idempotent \(e \in kG^*\) and a right \(\overline{kG^*}\)-comodule \(W\).
\end{theorem}

\subsection{Some conjectures}
\cref{construction:general2} provides examples of partial comodules, and by \cref{th:kG*} it completely describes the simple partial comodules if \(H = kG^*\). Since the category \(\mathsf{PMod}^{kG^*} \simeq {_{kG}} \mathsf{PMod}\) is semisimple if \(k\) is of characteristic \(0\) and algebraically closed, in fact all partial \(kG^*\)-comodules can be constructed from \cref{construction:general2}. We believe that this holds for more general Hopf algebras and that the construction describes at least the semisimple partial comodules. This involves in fact two conjectures: the first about simplicity and the second about the completeness of the construction. 

\setcounter{alphthm}{0}
\begin{alphconj}
    \label{conj:simple}
    Let \(H\) be a Hopf algebra and \(e \in H\) a subcentral idempotent such that the right coideal subalgebra it generates is finite-dimensional. 
    If \(W\) is a simple right \(\bar{H}_e\)-comodule, then \(W \cotensor^{\bar{H}_e} He\) is a simple partial \(H\)-comodule.
\end{alphconj}

We will show in \cref{se:groups} that this conjecture holds if \(H = kG\) for some finite group \(G\).

\begin{alphconj} 
    \label{conj:complete}
    Let \(H\) be a Hopf algebra. Then every finite-dimensional simple partial \(H\)-comodule is of the form
    \(W \cotensor^{\bar{H}_e} He\)
    for some subcentral idempotent \(e\) and simple right \(\bar{H}_e\)-comodule \(W\).
\end{alphconj}

In the next section, we will show that any \emph{\(1\)-dimensional} partial \(H\)-comodule is of that form, provided that \(H\) has invertible antipode.

Recall also the conjecture made in \cite[Conjecture 2.8]{ABVdilations}: if \(H\) is a semisimple Hopf algebra, then \(H_{par}\) is finite-dimensional and semisimple. We can combine this with \cref{conj:simple} and \cref{conj:complete} to obtain a proposed classification of partial comodules.

Suppose that \(k\) is an algebraically closed field of characteristic 0, and let \(H\) be a finite-dimensional cosemisimple Hopf algebra. Then \(H^*\) is a semisimple Hopf algebra, and \((H^*)_{par}\) is conjectured to be semisimple as well. Hence the category
\[\PMod^H \simeq {_{H^*}}\PMod \simeq {_{(H^*)_{par}}} \Mod\] 
is semisimple, whose simple objects are described by \cref{construction:general2}, if \cref{conj:complete} holds true. This allows us to describe the algebras \((H^*)_{par}\) and \(A_{par}(H^*)\) as follows. 

Let \(M\) be a finite-dimensional simple partial \(H\)-comodule. Then there is an algebra morphism
\begin{equation}
    \label{eq:thetaM}
    \theta_M \colon (H^*)_{par} \to \End_k(M), \qquad \theta_M\big([\varphi_1] \cdots [\varphi_n]\big) (m) = (M \otimes \varphi_1 \otimes \cdots \otimes \varphi_n) \rho^n(m),
\end{equation}
which is surjective by Jacobson density theorem (see e.\,g.\ \cite[Chapter XVII, \S3, Theorem 3.2]{Lang}). 

Let \(\mathcal{S}\) be the set of all (pairwise non-isomorphic) simple partial comodules created by \cref{construction:general2}. Putting together the algebra morphisms \eqref{eq:thetaM} gives a surjective algebra morphism
\begin{equation}
    \label{eq:theta}
    \theta : (H^*)_{par} \to \prod_{M \in \mathcal{S}} \End_k(M) \cong \prod_{M \in \mathcal{S}} \Mat_{n_M}(k)
\end{equation}
where \(n_M = \dim M\).
Since \((H^*)_{par}\) is supposed to be semisimple, we can expect from \cref{conj:complete} that \(\theta\) is an isomorphism of algebras. 

Recall the subalgebra \(A_{par}(H^*)\) of \((H^*)_{par}\) from \eqref{eq:Apar}:
\[A_{par}(H^*) = \langle \varepsilon_{\varphi} \mid \varphi \in H^* \rangle.\]
As always, let \(e\) be a subcentral idempotent and \(A\) the right coideal subalgebra generated by \(e\). Then \(e\) is central in \(A\) by \cref{le:ecentral} and \(Ae\) is a partial comodule algebra with unit \(e\), which arises from \cref{construction:general2} by taking \(W\) the trivial \(\bar{H}\)-comodule as explained in \cref{le:partialcomodulealgebra} and \cref{rk:comodulealgebra}. It is interesting to restrict the algebra map \(\theta_{Ae}\) \eqref{eq:thetaM} to \(A_{par}(H^*)\).

\begin{lemma}
	Let \(H\) be a finite-dimensional Hopf algebra and \(e \in H\) be a subcentral idempotent. Let also \(A\) be the right coideal subalgebra of \(H\) generated by \(e\). The map \(\theta_{Ae}\) (\ref{eq:thetaM}) restricts to an algebra map 
	\[A_{par}(H^*) \to \mathsf{End}_{A}(Ae) \cong Ae.\]
\end{lemma}
\begin{proof}
     The fact that \(\mathsf{End}_{A}(Ae) \cong Ae\) as algebras is standard: a right \(A\)-linear endomorphism \(f\) of \(Ae\) is uniquely determined by \(f(e) \in Ae\) and for every \(x \in Ae\), \(y \mapsto xy\) is a right \(A\)-linear endomorphism of \(Ae,\) which establishes the isomorphism.

	Recall that \(\rho(x) = x_{(1)} e \otimes x_{(2)} = ex_{(1)} \otimes x_{(2)}\) for \(x \in Ae\).
	Denote the antipode of \(H^*\) by \(S^*\). We calculate for \(\varphi \in H^*\)
	\begin{align*}
		\theta_{Ae}(\varepsilon_{\varphi})(x) = \theta_{Ae}([\varphi_{(1)}][S^*(\varphi_{(2)})])(x) &= ee_{(1)} x_{(1)} \varphi_{(1)}(e_{(2)}x_{(2)}) S^*(\varphi_{(2)})(x_{(3)}) \\
		&= ee_{(1)} x_{(1)} \varphi_{(1)}(e_{(2)}x_{(2)}) \varphi_{(2)}(S(x_{(3)})) \\
		&= ee_{(1)} x_{(1)} \varphi(e_{(2)}x_{(2)}S(x_{(3)})) \\
		&= ee_{(1)} \varphi(e_{(2)}) x
	\end{align*}
hence \(\theta_{Ae}(\varepsilon_{\varphi})\) is clearly right \(A\)-linear. 

Composing with the isomorphism \(\End_A(Ae) \cong Ae\), we obtain an algebra map
\begin{equation}
\label{eq:AparAe}
A_{par}(H^*) \to Ae : \varepsilon_\varphi \mapsto ee_{(1)} \varphi(e_{(2)}). 
\end{equation}
This concludes the proof.
\end{proof}

Let \(E(H)\) be the set of subcentral idempotents of \(H\). For \(e \in E(H)\), denote by \(A_e\) the right coideal subalgebra generated by \(e\). Putting together the morphisms \eqref{eq:AparAe} for all subcentral idempotents \(e\) of \(H\), we obtain an algebra morphism
\begin{equation}
    \label{eq:Aparconj}
    A_{par}(H^*) \to \prod_{e \in E(H)} A_e e : \varepsilon_{\varphi} \mapsto (ee_{(1)} \varphi(e_{(2)}))_{e \in E(H)}.
\end{equation}
\begin{alphconj}
    \label{conj:Apar}
    If \(H\) is finite-dimensional and cosemisimple, then the algebra morphism \eqref{eq:Aparconj} is an isomorphism. 
\end{alphconj}

\section{1-dimensional partial comodules}
\label{se:1D}

Let \(H\) be any Hopf algebra 
and take a \(1\)-dimensional partial comodule over \(H\). We will show that it is obtained by \cref{construction:general2}, if \(H\) has invertible antipode. A \(1\)-dimensional partial comodule is nothing else than a linear map
\[\rho : k \to H\]
satisfying the partial comodule axioms. Remark that \(\rho\) is completely determined by \(\rho(1) = r \in H\). The axioms \ref{PCM1}, \ref{PCM2} and \ref{PCM3} (and their variants \ref{PCM4}  and \ref{PCM5}), in terms of the element $r\in H$ give
\begin{subequations}\label{eq:PCMr}
\begin{align}
	& \text{\ref{PCM1}} \quad \epsilon(r) = 1 ; \label{PCM1r}\\
	& \text{\ref{PCM2}} \quad r\otimes rS(r)=r_{(1)}\otimes r_{(2)}S(r) ; \label{PCM2r} \\
	& \text{\ref{PCM3}} \quad rS(r_{(1)})\otimes r_{(2)} =rS(r)\otimes r ; \label{PCM3r} \\
	& \text{\ref{PCM4}} \quad r \otimes S(r) r = r_{(1)} \otimes S(r_{(2)}) r ; \label{PCM4r}\\
	& \text{\ref{PCM5}} \quad S(r) r \otimes r = S(r) r_{(1)} \otimes r_{(2)} \label{PCM5r}.
 \end{align}
\end{subequations}
From equation (\ref{PCM4r}), it  follows that 
\begin{equation}\label{rS(r)r}
	r S(r) r = r_{(1)} S(r_{(2)}) r = \epsilon(r) r = r,
\end{equation}
and from (\ref{PCM5r}), we also have 
\begin{equation}\label{S(r)rS(r)}
	S(r) r S(r) = S(r) r_{(1)} S(r_{(2)}) = S(r) \epsilon(r) = S(r) .
\end{equation}
This means that the element $r\in H$ is a von Neumann regular element.

\begin{proposition}\label{pr:Sofr}
    Suppose that \(r \in H\) satisfies all equations \eqref{eq:PCMr} and that \(g \in H\) is a grouplike element. Then \(S(r)\) and \(gr\) satisfy these equations as well. If moreover \(S\) is invertible, then \(S^{-1}(r)\) satisfies them, too.
\end{proposition}

\begin{proof}
    The proof is a straightforward check. For example, \eqref{PCM5r} for \(r\) entails \eqref{PCM2r} for \(S^{\pm1}(r)\) by taking $\tau \circ (S \otimes S)$ on both sides, where $\tau \colon H\otimes H \rightarrow H\otimes H$ is the usual flip operator:
    \[S(r)_{(1)} \otimes S(r)_{(2)}S\big(S(r)\big) = S(r_{(2)}) \otimes S\big(S(r)r_{(1)}\big) \stackrel{\eqref{PCM5r}}{=} S(r) \otimes S(r)S\big(S(r)\big)\]
    and similarly for \(S^{-1}(r)\). Moreover, the fact that \(gr\) also defines a partial comodule follows directly from \cref{ex:gtrans}.
\end{proof}

\begin{proposition}
	\label{pr:ptensorp}
	Let \(r \in H\) be an element satisfying \eqref{eq:PCMr} and let \(e \coloneqq rS(r)\). Then \(e^2 = e\) and
	\begin{equation}\label{eq:pp1op2}
 e e_{(1)} \otimes e_{(2)} = e \otimes e = e_{(1)}e \otimes e_{(2)}.
 \end{equation}
\end{proposition}

\begin{proof}
	From the identity \eqref{rS(r)r} we conclude that $rS(r)rS(r) = rS(r)$. Therefore, $e$ is idempotent. 
	
	Consider now the expression of $ee_{(1)}\otimes e_{(2)}$:
	\begin{eqnarray*}
		ee_{(1)} \otimes e_{(2)} & = &  rS(r)r_{(1')} S(r)_{(1)} \otimes r_{(2')}S(r)_{(2)} \\
		& \stackrel{\eqref{PCM5r}}{=} & rS(r)r S(r_{(2)}) \otimes rS(r_{(1)}) \\
		& \stackrel{\eqref{PCM3r}}{=} & rS(r) rS(r) \otimes rS(r) \\
		& \stackrel{\eqref{rS(r)r}}{=} & rS(r) \otimes rS(r) \\
		& = & e\otimes e.
	\end{eqnarray*}
	On the other hand,
	\begin{eqnarray*}
		e_{(1)} e \otimes e_{(2)} & = & r_{(1')}S(r_{(2)})rS(r) \otimes r_{(2')}S(r_{(1)}) \\
		& \stackrel{\eqref{PCM4r}}{=} & r_{(1')}S(r)r S(r) \otimes r_{(2')}S(r) \\
		& \stackrel{\eqref{PCM2r}}{=} & rS(r) r S(r) \otimes rS(r) \\
		& \stackrel{\eqref{rS(r)r}}{=} & rS(r) \otimes rS(r) \\
		& = & e \otimes e,
	\end{eqnarray*}
	as claimed.
\end{proof}

\begin{corollary}\label{cor:sinvsubcentr}
 If \(S\) is invertible, then \(e = S^{-1}(r)r\) is a subcentral idempotent satisfying \eqref{eq:pp1op2}.
\end{corollary}

\begin{proof}
    It follows from \cref{pr:ptensorp} applied to $S^{-1}(r)$, in view of \cref{pr:Sofr}.
\end{proof}

\cref{pr:ptensorp} shows that we can apply \cref{construction:general2} on the idempotent \(e = rS(r)\) for \(r \in H\) satisfying \eqref{eq:PCMr}. So, let \(A\) be the right coideal subalgebra generated by \(e,\) i.\,e.\ the subalgebra generated by \(H^* \cdot e = \{e_{(1)} \varphi(e_{(2)}) \mid \varphi \in H^*\}\). Recall that \(HA^+ = H(A \cap \ker \epsilon)\) is a coideal in \(H\), so that \(\bar{H} = H/HA^+\) is a coalgebra. Denote by \(\pi\) the projection \(H \to \bar{H}\), which is, by construction, a morphism of coalgebras.

\begin{lemma}\label{lem:grplike}
	Using the foregoing notation, \(\pi\big(S(r)\big)\) is a grouplike element in \(\bar{H}\).
\end{lemma}
\begin{proof}
	Indeed, since $\pi :H\rightarrow \bar{H}$ is a morphism of coalgebras,
	\begin{eqnarray*}
		\Delta (\pi (S(r))) & = & \pi (S(r)_{(1)}) \otimes \pi (S(r)_{(2)}) \\
		& \stackrel{(\star)}{=} & \pi (S(r_{(2)}) e) \otimes \pi (S(r_{(1)})) \\
		& = & \pi (S(r_{(2)}) rS(r)) \otimes \pi (S(r_{(1)})) \\
		& \stackrel{\eqref{PCM4r}}{=} & \pi (S(r) rS(r)) \otimes \pi (S(r)) \\
		& = & \pi (S(r)) \otimes \pi (S(r)) .
	\end{eqnarray*}
	Here, in equality \((\star)\) we used that $e\in A$ and $\epsilon (e)=\epsilon (rS (r)) =1$, so that for any \(h \in H,\) $he - h \in HA^+$, which implies that $\pi (h)=\pi (h e)$. Therefore, $\pi (S(r))\in \bar{H}$ is grouplike.
\end{proof} 

\begin{lemma}
	\label{le:p_int}
 Consider \(e \in H\) 
 satisfying \eqref{eq:pp1op2} and let \(A\) be the coideal subalgebra generated by \(e\). Let \(\pi : H \to H/HA^+\) be the canonical coalgebra projection as before.
	\begin{enumerate}[(i),leftmargin=0.8cm]
		\item Then \(e\) is an integral for \(A\), i.\,e.\ \(ae = \epsilon(a)e = ea\) for all \(a \in A\). 
		\item If \(\pi(x) = \pi(y)\), then \(xe = ye\). 
	\end{enumerate}
\end{lemma}
\begin{proof}
	\begin{enumerate}[(i),leftmargin=0.8cm]
	\item A typical element of $A$ is a linear combination of products of the type $e_{(1)}\varphi (e_{(2)})$, for $\varphi \in H^*$. Applying \(H \otimes \varphi\) to \eqref{eq:pp1op2}, we get
	\[ee_{(1)}\varphi (e_{(2)}) = e_{(1)}\varphi(e_{(2)})e  = \varphi(e)e = \epsilon(e_{(1)}\varphi(e_{(2)})) e.\]
Therefore, $e$ is a left and right integral in $A$.

	\item If $\pi (x) =\pi (y)$, then $x-y \in HA^+$, that is, 
	\[
	x-y = \sum_i h_i a_i , \qquad h_i \in H, \quad a_i \in A \quad \text{ and  } \quad \epsilon (a_i )=0.
	\]
	Hence, 
	\[
	xe-ye =(x-y)e =\sum_i h_i a_i e \overset{(i)}{=}\sum_i h_i \epsilon (a_i )e =0 ,
	\]
	which means that $xe=ye$. \qedhere
	\end{enumerate}
\end{proof}

We are now ready to prove our main theorem.

\begin{alphthm}
    \label{th:1D}
	Every \(1\)-dimensional partial comodule over a Hopf algebra \(H\) with invertible antipode is of the form \(W  \cotensor^{\bar{H}_e} He\) for some subcentral idempotent \(e \in H\) satisfying \eqref{eq:pp1op2} and some \(1\)-dimensional right \(\bar H_e\)-comodule \(W\).
\end{alphthm}
\begin{proof}
	The partial comodule structure on \(k\) is completely determined by \(\rho(1) = r \in H,\) and by \cref{cor:sinvsubcentr}, \(e = S^{-1}(r)r\) is a subcentral idempotent (satisfying \eqref{eq:pp1op2}).
	
	In view of \cref{lem:grplike} applied to \(S^{-1}(r)\), \(\pi(r)\) is grouplike in \(\bar H\). Let then \(W\) be the \(1\)-dimensional right \(\bar{H}\)-comodule \(1_k \mapsto 1_k \otimes \pi(r)\). Recall that by \cref{le:Hbarcomodule}, \(He\) is a left \(\bar{H}\)-comodule via \(he \mapsto \pi(h_{(1)}) \otimes h_{(2)}e\). Hence
	\begin{equation}
		\label{eq:WHp1D}
		W \cotensor^{\bar{H}} He = \{he \in He \mid \pi(h_{(1)}) \otimes h_{(2)} e = \pi(r) \otimes he\}.
	\end{equation}
	By applying \(\epsilon\) on the right tensorand of the condition defining (\ref{eq:WHp1D}), we find that if \(he \in W \cotensor^{\bar{H}} He\), then
	\[\pi(h) = \pi(\epsilon(h) r).\]
	By \cref{le:p_int}, we conclude that if \(he \in W \cotensor^{\bar{H}} He\), then \(he = \epsilon(h) re = \epsilon(h) r S^{-1}(r) r = \epsilon(h) r\). 
	Hence \(W \cotensor^{\bar{H}} He\) is \(1\)-dimensional: it is the linear span of \(r\).
	
	Finally, the partial comodule structure on \(W \cotensor^{\bar{H}} He = \langle r \rangle\) is given by
	\[r \mapsto r_{(1)} e \otimes r_{(2)} = r_{(1)} S^{-1}(r) r \otimes r_{(2)} \overset{\eqref{PCM3r}}{=} r S^{-1}(r) r \otimes r = r \otimes r\]
	so \(W \cotensor^{\bar{H}} He\) is indeed isomorphic to the \(1\)-dimensional partial comodule we started with.
\end{proof}

If \((k, \rho)\) is not just a partial comodule but a partial comodule algebra, then it follows from the axiom \ref{PC1} that \(\epsilon(r) = 1\), from \ref{PC2} that \(r\) is already idempotent and from \ref{PC3} that \(r\) is subcentral and satisfies \(r_{(1)} r \otimes r_{(2)} = r \otimes r\). Then 
\[e = S^{-1}(r) r = S^{-1}(r_{(2)}) r_{(1)} r = \epsilon(r) r = r\]
and \(\pi(r) = \pi(1_H)\). So using the notation of \cref{le:partialcomodulealgebra}, the \(1\)-dimensional partial comodule algebras of \(H\) are all of the form \(T \cotensor^{\bar{H}} He \cong {^{\mathrm{co}\bar{H}_e}}H \, e\) for some subcentral idempotent \(e\).

If \(H\) is finite-dimensional, then we are in the setting of \cref{rk:comodulealgebra} and \({^{\mathrm{co}\bar{H}_e}}H \, e = A_e\, e\). In that case \(H^*\) is a Hopf algebra too, and a \(1\)-dimensional partial \(H^*\)-comodule is a \(1\)-dimensional partial \(H\)-module. In \cite{MPS}, \(1\)-dimensional partial module algebras are studied. By the foregoing discussion, these correspond to subcentral idempotents in \(H^*\) (i.\,e.\ a partial module algebra structure \(\lambda : H \to k\) on \(k\) is a subcentral idempotent of \(H^*\)). The partial smash product algebra \(k \underline{\#} H^*\) (see \cite[\S4]{MPS}) is
\[\{\varphi_{(1)}(r) \varphi_{(2)} \mid \varphi \in H^*\}\]
which is a right coideal subalgebra of \(H^*\) associated to the subcentral idempotent \(r \in H\).

\begin{example}
    Consider Sweedler's Hopf algebra \(H_4\) over a field \(k\) of characteristic different from \(2\). This is a \(4\)-dimensional Hopf algebra with basis \(\{1, g, x, gx\}\) where 
    \begin{align*}
        g^2 &= 1, & x^2 &= 0; & xg &= -gx; \\
        \Delta(g) &= g \otimes g, & \epsilon(g) &= 1, & S(g) &= g; \\
        \Delta(x) &= 1 \otimes x + x \otimes g, & \epsilon(x) &= 0, & S(x) &= gx.
    \end{align*}

    We want to determine all the \(1\)-dimensional right partial comodules over \(H = H_4\), up to isomorphism. By \cref{th:1D}, we know that they are all of the form \(W  \cotensor^{\bar{H}} He\) for some subcentral idempotent \(e \in H\) satisfying \eqref{eq:pp1op2} and some \(1\)-dimensional right \(\bar H\)-comodule \(W\). By a direct computation, the subcentral idempotents in \(H_4\) are \(1\) and \(\frac{1 \pm g}{2} + \gamma x\) for \(\gamma \in k\) and only those of the form \(1\) and \(\frac{1 + g}{2} + \gamma x\) satisfy \eqref{eq:pp1op2}. For each of them, the ideal \(Ae = \prescript{\mathrm{co}\bar H}{}{H}\,e\) is \(1\)-dimensional, generated by \(e\), and it is a partial \(H\)-comodule with respect to
    \[\rho \colon e \mapsto e_{(1)}e \otimes e_{(2)} \stackrel{\eqref{eq:pp1op2}}{=} e \otimes e.\]
    For \(e = 1\), the quotient coalgebra is \(\bar{H} = H,\) for \(e = \frac{1 + g}{2} + \gamma x\) with \(\gamma \neq 0\) we get a two-dimensional coalgebra which has a basis of grouplikes \(\{\pi(1), \pi(g)\}\). This means that in each of these cases, there are two choices for \(W\) to construct a partial comodule using \cref{construction:general2}. If \(\gamma = 0\), i.\,e.\ if \(e = \frac{1 + g}{2}\), then \(\bar{H}\) contains only one grouplike: \(\pi(1)\). 
    In any one of these cases, we are in the setting of \cref{prop:Ierhog} and hence any additional \(1\)-dimensional partial comodule can be obtained from \(Ae\) by modifying \(\rho\) into \(\rho^g\).
    Thus, a complete list of \(1\)-dimensional partial \(H_4\)-comodules is
    \begin{align*}
        &(\text{from } e = 1 :) & 1 &\mapsto 1 \otimes 1, & 1 &\mapsto 1 \otimes g, \\
        &(\text{from } e = \frac{1 + g}{2} :) & 1 & \mapsto 1 \otimes \frac{1 + g}{2}, && \\
        &(\text{from } e = \frac{1 + g}{2} + \gamma x \text{ with } \gamma \neq 0 :) & 1 & \mapsto 1 \otimes \left(\frac{1 + g}{2} + \gamma x\right), & 1 & \mapsto 1 \otimes \left(\frac{1 + g}{2}  + \gamma gx\right). 
        \end{align*}

    In fact, \(H_{4, par}\) has been explicitly described in \cite{ABVparreps}. From this description it can be deduced that every simple partial \(H_4\)-module is \(1\)-dimensional. Since \(H_4\) is self-dual, also every simple partial comodule is \(1\)-dimensional. 
\end{example}

\begin{remark}
    One could try to imitate the proof of \cref{th:1D} in the general case of finite-dimensional partial comodules.
    If we start from an \(n\)-dimensional partial comodule \(M\) rather than a \(1\)-dimensional one, the partial coaction \(\rho : M \to M \otimes H\) is no longer determined by a single element of \(H\) but by a matrix with entries in \(H\). Choosing a basis \(\{m_1, \dots, m_n\}\) for \(M\), we can write
    \[\rho(m_i) = \sum_{j = 1}^n m_j \otimes r_{ij} \in M \otimes H\]
    and obtain matrices
    \begin{align*}
        R &= (r_{ij})_{1 \leq i, j \leq n}, \\
        S(R) &= (S(r_{ij}))_{1 \leq i, j \leq n}, \\
        P &= R\ S(R)
    \end{align*}
    in \(\mathsf{Mat}_n(H)\).
    This matrix \(P\) is idempotent and satisfies \(PP_{(1)} \otimes P_{(2)} = P_{(1)} P \otimes P_{(2)}\) (if interpreted appropriately as an identity in \(\mathsf{Mat}_n(H \otimes H)\)), but it is unclear how to extract a subcentral idempotent \(e\) which recovers the original partial comodule. 
\end{remark}

\section{Partial comodules of finite groups}
\label{se:groups}
In this section we will apply \cref{construction:general2} to \(H = kG\) for some finite group \(G\) such that \(\mathrm{char}(k) \nmid |G|\). We will show that the construction gives a simple partial comodule if \(W\) is simple (thus proving \cref{conj:simple} in this particular case) and we determine when two partial comodules obtained from the construction are isomorphic.  

\subsection{Simplicity and redundancy}

First we will look at partial comodule algebras as described in \cref{le:partialcomodulealgebra} and \cref{rk:comodulealgebra}.
The right coideal subalgebras of \(kG\) are exactly the Hopf subalgebras, and those are of the form \(kK\) for a subgroup \(K \leq G\). 

So, a subcentral idempotent \(e = \sum_{g \in G} \mu_g\, g\in kG\) generates a right coideal subalgebra \(kK,\) where \(K\) is exactly the subgroup generated by \(\{g \in G \mid \mu_g \neq 0\}\). 
Since \(e\) is a central idempotent of \(kK\), it generates a representation \(I = kKe\) 
\[e = \frac{1}{|K|} \sum_{h \in K} \chi(h^{-1}) h\]
where \(\chi\) is the character of \(I\),  
as the following well-known fact states.

\begin{lemma}
\label{le:eformula}
Let \(K\) be a finite group and \(k\) be a field such that \(\mathrm{char}(k) \nmid |K|\). If \(e\) is a central idempotent in \(kK\) and \(I = kKe\), then 
\[e = \frac{1}{|K|} \sum_{h \in K} \chi(h^{-1}) h\]
where \(\chi\) is the character of \(I\).
In particular, if \(e_V\) is the central idempotent corresponding to the identity endomorphism of an irreducible representation \(V\) of \(K\), then
\[e_V = \frac{\chi_V(1)}{|K|} \sum_{h \in K} \chi_V(h^{-1}) h,\]
where \(\chi_V\) is the corresponding irreducible character.
\end{lemma}

\begin{proof}
    Let \(\Lambda\colon kK \to \End_k(kK)\) denote the regular representation and \(\lambda \colon kK \to \End_k(I)\) denote the one on \(I\). Write \(e = \sum_{g \in K}\alpha_g\, g\). Choose a basis \(\{v_i,w_j\mid i \in E,j \in F\}\) of \(kK\) where \(\{v_i\mid i \in E\}\) is a basis of \(I\) and \(\{w_j\mid j \in F\}\) is a basis of \(kK(1-e)\). For all \(h \in K\) we have on the one hand
    \[\mathsf{tr}(\Lambda_{h^{-1}e}) = \sum_{g \in K} \alpha_g \mathsf{tr}(\Lambda_{h^{-1}g}) = \alpha_h|K|\]
    and on the other hand
    \[\mathsf{tr}(\Lambda_{h^{-1}e}) = \sum_{i \in E} p_{v_i}\big(\Lambda_{h^{-1}e}(v_i)\big) + \sum_{j \in F} p_{w_j}\big(\Lambda_{h^{-1}e}(w_j)\big) = \sum_{i \in E} p_{v_i}\big(\lambda_{h^{-1}}(v_i)\big) = \chi(h^{-1}).\]
    Therefore, \(\alpha_h = \chi(h^{-1})/|K|\) for all \(h \in K\). The last claim follows from observing that if \(e_V\) is the central idempotent corresponding to \(V\), then \(I \cong \dim(V)V\) and so \(\chi = \chi_V(1)\chi_V\).
\end{proof}

We will now show that \(I\) is simple as a partial comodule. Recall that the partial coaction on \(I\) is given by
\begin{equation}
    \label{eq:Icoaction}
    \rho \colon I \to I \otimes H : x \mapsto x_{(1)} e \otimes x_{(2)}.
\end{equation}

\begin{lemma}\label{lem:subcntsimple}
	Let \(e = \sum_{g \in G} \mu_g\, g\) be a subcentral idempotent in \(kG\) and let \(K\) be the subgroup generated by \(J = \{g \in G \mid \mu_g \neq 0\}\). Then the partial comodule \(I = kK e\) is simple.
\end{lemma}

\begin{proof}
	Let \(M \subseteq I\) be a non trivial partial subcomodule. Without loss of generality, \(e \in M\). Indeed, if \(x = \sum_{g \in K} \lambda_g\, g \in M \setminus \{0\},\) then \(\rho(x) = \sum_{g \in K} \lambda_g \, ge \otimes g\) and we see that \(he \in M\) whenever \(\lambda_h \neq 0\) because \(M\) is a partial subcomodule. Let \(h \in K\) be such that \(\lambda_h \neq 0\). It is easy to check that \(h^{-1}M\) is a non trivial partial subcomodule of \(I\), and \(h^{-1}x = \lambda_h \, 1_G + \sum_{g \neq 1_G} \lambda_{gh}\, g\). It follows that \(e \in h^{-1} M\). 
	
	We will show that \(I = M\). Since \(e \in M,\) also \(he \in M\) for all \(h \in J\). We have
	\[he = \sum_{g \in K} \mu_g \, hg,\]
	thus it follows that \(hh' e \in M\) for all \(h, h' \in J\). By induction, all products of elements in \(J\) (i.\,e.\ all elements of \(K\) because \(J\) is generating) have the property that \(h_1 \cdots h_n e  \in M,\) which shows that \(M = kKe  = I\). We conclude that \(I\) is simple. 
\end{proof}

\begin{remark}
Observe that \(I\) need not be simple as a \(K\)-representation. For instance, \(e = 1_G - \frac{1}{|G|} \sum_{g \in G} g\) is a subcentral idempotent for which \(kGe\) is an \((n - 1)\)-dimensional simple partial comodule, but of course the \(G\)-representation corresponding to \(e\) is not simple if \(|G| > 2\).
\end{remark}

To continue the construction, remark that
\[kG\, (kK)^+ = \left\langle gh - g \mid g \in G, h \in K \right\rangle \]
and
\begin{gather*}
    \overline{kG} = kG/(kG\,(kK)^+) \cong k(G/K), \\
    \pi : kG \to k(G/K) : g \mapsto gK
\end{gather*}
where \(G/K = \{g_i K \mid i = 1, \dots, n\}\) is the set of left cosets of \(K\) in \(G\). The coalgebra \(k(G/K)\) is cocommutative and has a basis of grouplikes \(G/K\).

Every simple right \(k(G/K)\)-comodule is \(1\)-dimensional; let \(W_i\) be the simple comodule defined by the grouplike \(\pi(g_i) = g_iK\). 
Then
\begin{equation}
    \label{eq:WikG}
    W_i   \cotensor^{\overline{kG}} kG e \cong \big(I, \rho_i\big),
\end{equation}
where as before \(I=kKe\), and
\begin{equation}
    \label{eq:Irhoi}
    \rho_i \colon I \to I \otimes H : x \mapsto x_{(1)} e \otimes g_i x_{(2)},
\end{equation}
by \cref{prop:Ierhog}.
This shows that \(W_i \cotensor^{\overline{kG}} kG  e\) is simple too, because its partial subcomodules are in one-to-one correspondence with the partial subcomodules of \((I,\rho_i)\), which in turn are in one-to-one correspondence with those of \((I,\rho)\). 
To summarize, we have the following theorem.

\begin{theorem}
    \label{th:groupcase_simple}
    If \(H = kG\) for a finite group \(G\) and \(e\) is a subcentral idempotent in \(kG\), then the partial comodule \(W \cotensor^{\overline{kG}} kG e\) obtained by \cref{construction:general2} is simple if and only if \(W\) is a simple (i.\,e.\ \(1\)-dimensional) \(\overline{kG}\)-comodule.
\end{theorem}

\begin{proof}
    We have proved in the foregoing paragraphs that if \(W\) is a simple \(\overline{kG}\)-comodule, then \(W \cotensor^{\overline{kG}} kG e\) is simple as well. 
    To prove the converse, suppose that for a certain \(\overline{kG}\)-comodule $W$, the partial \(kG\)-comodule \(W \cotensor^{\overline{kG}} kG e\) is simple. By cosemisimplicity of \(\overline{kG} \cong k(G/K)\), \(W\) contains a simple (hence 1-dimensional) subcomodule \(kw\), and there exists a subcomodule \(U\) such that \(W\) is the biproduct \(kw \oplus U\). In view of the isomorphisms \eqref{eq:isomodcomod}, we know that \(\PMod^{kG}\) is abelian and, in particular, it has biproducts as well, which are preserved by the forgetful functor to \(k\)-vector spaces. Therefore, 
    \[W \cotensor^{\overline{kG}} kG e \cong \big(kw \oplus U\big) \cotensor^{\overline{kG}} kG e \cong \Big(kw \cotensor^{\overline{kG}} kG e\Big) \oplus \Big(U \cotensor^{\overline{kG}} kG e\Big).\]
    Now, \(kw \cotensor^{\overline{kG}} kG e\) cannot be \(0\): since \(kw\) is a 1-dimensional \(k(G/K)\)-comodule, we have that \(w^{(0)} \otimes w^{(1)} = w \otimes gK\) for some $g\in G$, hence \(0\neq w \otimes ge \in kw \cotensor^{\overline{kG}} kG e\). Thus, \(U \cotensor^{\overline{kG}} kG e = 0\) by simplicity, which entails that \(U = 0\) by repeating the reasoning from the previous sentence.
\end{proof}

Finally, we study when two of these partial comodules are isomorphic.

\begin{theorem}
	\label{th:redundancyI}  
    Let \(e\) and \(e'\) be subcentral idempotents in \(kG\).
	The two simple partial comodules \(M = W \cotensor^{k(G/K)} kGe\) and \(M' = W' \cotensor^{k(G/K')} kGe'\) are isomorphic if and only if \(K = K', W \cong W'\) and there exists a multiplicative character \(\nu\) of \(K\) such that the algebra automorphism
	\begin{equation}
		\label{eq:characteraction}
		\Phi\colon kK \to kK , \quad g \mapsto \nu(g) g
	\end{equation}
	maps \(e\) to \(e'\).
\end{theorem}

\begin{proof}
    The quotient coalgebra \(k(G/K)\) is cosemisimple because it has a basis of grouplikes, and the cotensor product \(W \cotensor^{k(G/K)} kG e\) is never zero (see \eqref{eq:WikG}). Furthermore, \(W\) and \(W'\) need to be \(1\)-dimensional: if not, then \(M\) and \(M'\) would not be simple. 
    So \(W \cotensor^{\bar{H}} He \cong kKe =: I\) as vector spaces and by \eqref{eq:Irhoi}, \(W \cotensor^{\bar{H}} He \cong (I, \rho_i)\) as partial comodules for some \(g_i\), where \(\rho_i\) was defined in \eqref{eq:Irhoi}. Similarly \(W' \cotensor^{\bar{H'}} He' \cong (I', \rho_j)\), where \(\rho_j(h'e') = h'e'_{(1)}e' \otimes g_jh'e'_{(2)}\) for some \(g_j\in G\) and for all \(h'\in K'\). Remember from \cref{lem:subcntsimple} that we denote by \(J\subseteq G\) the set of \(g \in G\) appearing in the expression of \(e\) with non-zero coefficient and by \(K\) the subgroup of \(G\) generated by \(J\), so that \(e = \sum_{g \in J} \alpha_g \, g\). Then for any \(h \in K\)
    \[\rho_i(he) = he_{(1)} e \otimes g_ih e_{(2)} = \sum_{g, g' \in J} \alpha_g \alpha_{g'} hgg' \otimes g_i hg\]
    and so
    \[(p_{hg^2} \otimes kG)\big(\rho_i(he)\big) = \alpha_g^2 g_i hg,\]
    where \(p_g\colon kG \to k\) is the dual basis element corresponding to \(g \in G\).
    This entails that
    \[\left\{(\varphi \otimes kG)\big( \rho_i(x)\big) \mid x \in I, \varphi \in kG^*\right\} = k (g_i K).\]
    This space is invariant under isomorphism of partial comodules. Therefore, \(k(g_iK) = k(g_jK')\) as subspaces of the coalgebra \(kG\) and hence \(g_iK = g_jK'\), from which we conclude that \(K = K'\). Moreover, \(W \cong W'\) because both are \(k\) with the \(1\)-dimensional comodule structure coming from the grouplike \(g_i K \in G/K\). 
    We are left to determine the character \(\nu\colon K \to k\). To this aim, denote by \(f \colon (I,\rho_i) \to (I',\rho_j)\) the isomorphism of partial comodules existing by hypothesis. From \(g_iK = g_jK\) it follows that \(I'= kK'e' = kKe'\) with the partial coaction \(\rho_j(x) = x_{(1)}e'\otimes g_jx_{(2)}\) as above is isomorphic to \(I'\) itself with the partial coaction \(\rho_i(x) = x_{(1)}e' \otimes g_ix_{(2)}\) via the automorphism \(f_{i,j}\colon I'\to I', x \mapsto g_j^{-1}g_ix\). In fact,
    \[\rho_j\big(f_{i,j}(x)\big) = \rho_j\big(g_j^{-1}g_ix\big) = g_j^{-1}g_ix_{(1)}e' \otimes g_ix_{(2)} = f_{i,j}(x_{(1)}e') \otimes g_ix_{(2)} = \big(f_{i,j} \otimes kG\big)\big(\rho_i(x)\big).\]
    By composing the two isomorphisms \((I,\rho_i) \xrightarrow{f} (I',\rho_j) \xrightarrow{f_{i,j}} (I',\rho_i)\), we obtain an isomorphism \(\varphi  \colon (I,\rho_i) \to (I',\rho_i)\) of partial comodules. The following computation:
    \begin{align*}
    \big(\varphi  \otimes kG\big)\big(\rho_1(x)\big) & = \big(\varphi  \otimes kG\big)\big((1_k \otimes g_i^{-1})\cdot\rho_i(x)\big) = (1_k \otimes g_i^{-1})\cdot\big(\varphi  \otimes kG\big)\big(\rho_i(x)\big) \\
    & = (1_k \otimes g_i^{-1})\cdot\rho_i\big(\varphi (x)\big) = \rho_1\big(\varphi (x)\big),
    \end{align*}
    shows that the partial comodules \((I, \rho_1)\) and \((I', \rho_1)\) are isomorphic, too, via the very same \(\varphi \). 
	Thus, let \(\varphi \colon I \to I'\) be the resulting isomorphism of partial comodules and let \(x \in I\) be arbitrary. Then
	\begin{equation}
		\label{eq:comodulemor}
		\varphi(x_{(1)}e) \otimes x_{(2)} =  \varphi(x)_{(1)} e' \otimes \varphi(x)_{(2)}.
	\end{equation}
	If \(x = \sum_{g \in K} p_g(x) g\) and \(\varphi(x) = \sum_{g \in K} p_g(\varphi(x)) g\), (\ref{eq:comodulemor}) implies that for all \(g\)
	\begin{equation}
		\label{eq:phieg1}
		p_g(x) \varphi(ge) = p_g(\varphi(x)) ge'.
	\end{equation}
	Remark that both \(ge\) and \(ge'\) are non zero and that \(\varphi(ge) \neq 0\) because \(\varphi\) is an isomorphism. It follows that \(p_g(x) = 0\) if and only if \(p_g(\varphi(x)) = 0\).
	Since for any \(g \in K,\) there is an \(x_g \in I\) with \(p_g(x_g) \neq 0,\) we can conclude that for all \(g \in K\)
	\begin{equation}
		\label{eq:phieg2}
	\varphi(ge) = \nu_g\, ge'
	\end{equation}
	where \(\nu_g \coloneqq p_g(\varphi(x_g)) / p_g(x_g) \in k\). From \eqref{eq:phieg1} and \eqref{eq:phieg2} we have then
	\begin{equation}
		\label{eq:nu}
		p_g(\varphi(x)) = \nu_g\, p_g(x)
	\end{equation}
 	for all \(x \in I\) and all \(g \in K\).	
	Without loss of generality, we may suppose that \(\nu_1 = 1\): indeed, it suffices to replace \(\varphi\) with \(\nu_1^{-1} \varphi\), since \(\nu_1 \neq 0\) because \(\varphi\) is injective. Thus, \(\varphi(e) = e'\) and if we write again \(e = \sum_{g \in J} \alpha_g \, g\), then \(e' = \sum_{g \in J} \nu_g \alpha_g \, g\) by (\ref{eq:nu}). Taking \(x = he\) for \(h \in K\) arbitrary, the left-hand side of \eqref{eq:comodulemor} becomes
	\[\sum_{g \in J} \alpha_g \, \varphi(hge) \otimes hg = \sum_{g \in J} \alpha_g \nu_{hg} \, hge' \otimes hg\]
	and the right-hand side (using \eqref{eq:phieg2})
	\[\nu_h (he')_{(1)} e' \otimes (he')_{(2)} = \sum_{g \in J} \nu_h \nu_g \alpha_g \, hge' \otimes hg.\]
	Hence for all \(h \in K\) and \(g \in J\)
	\[\alpha_g \nu_{hg} = \alpha_g \nu_h \nu_g\]
	and so, since \(\alpha_g \neq 0\) for \(g \in J\), \(\nu_{hg} = \nu_h \nu_g\). But \(J\)
	generates \(K\) and in this way it is easy to see that \(\nu_{gh} = \nu_g \nu_h\) for all \(g, h \in K\). We have shown that \(\nu\colon K \to k^\times, g \mapsto \nu_g,\) defines a one-dimensional representation of \(K\) and so \eqref{eq:characteraction} is a well-defined algebra morphism which moreover maps \(e\) to \(e'\).
	
	Conversely, let \(\nu\) be a one-dimensional representation of \(K\) such that
	\[\Phi \colon kK \to kK : g \mapsto \nu(g) g\]
	maps \(e\) to \(e'\). Restricting gives a bijection \(\varphi \colon I \to I'\) (its inverse is the restriction of \(kK \to kK : g \mapsto \nu(g)^{-1} g\)). It is a morphism of partial comodules because for any \(h \in K,\)
	\begin{align*}
 \varphi(he_{(1)}e) \otimes he_{(2)} & = \sum_{g \in J}\alpha_g\Phi(hge) \otimes hg = \sum_{g \in J}\alpha_g\Phi(hg)\Phi(e) \otimes hg = \sum_{g \in J}\alpha_g\nu(hg) hg e' \otimes hg \\
 &=  \varphi(he)_{(1)} e'  \otimes \varphi(he)_{(2)}.
 \end{align*}
 If moreover \(K = K'\) and \(W \cong W'\), then clearly \(M \cong M'\) as partial comodules.	
\end{proof}

\begin{remark}
\label{rk:actioncharacter}
	The set of \(1\)-dimensional representations of \(K\) forms an abelian group which is isomorphic to \(\widehat{K^{\mathrm{ab}}},\) the Pontryagin dual of the abelianisation of \(K\). Indeed, given a 1-dimensional representation \(\nu : K \to k,\) we have \([K, K] \subseteq \ker \nu,\) so \(\nu\) induces a representation
	\[\overline{\nu} : K^{\mathrm{ab}} \to k.\]
	In particular there are exactly \([K : [K, K]] = |\widehat{K^{\mathrm{ab}}}|\) \(1\)-dimensional \(K\)-representations.

    Now \(\widehat{K^{\mathrm{ab}}}\) acts on the set of central idempotents of \(kK\) by \(\nu \cdot e = \Phi_\nu(e)\) where \(\Phi_\nu\) is defined as in \eqref{eq:characteraction}. Recall from \cref{rk:comodulealgebra} that a subcentral idempotent \(e\) in \(kG\) gives rise to a partial comodule algebra \(Ae\). \cref{th:redundancyI} says that subcentral idempotents \(e\) and \(e'\) produce isomorphic partial comodule algebras if and only if they generate the same right coideal subalgebra \(kK,\) and are in the same orbit under the action of \(\widehat{K^{\mathrm{ab}}}\) just described. 
\end{remark}

We have proven in \cref{th:1D} that all 1-dimensional partial comodules over an arbitrary Hopf algebra are obtained from our general construction. In case of a group algebra, we can combine this with the results from this section, to arrive at the following classification of 1-dimensional partial comodules over a group algebra:

\begin{corollary}
    Let \(G\) be a finite group. Then every 1-dimensional partial \(kG\)-comodule is of the form \eqref{eq:partial_integral} or \eqref{eq:partial_integral_trans} as described in \cref{ex:partial_integral}.
\end{corollary}
\begin{proof}
By \cref{th:1D}, all \(1\)-dimensional partial \(kG\)-comodules can be constructed using \cref{construction:general2}, so they are of the form \eqref{eq:Irhoi}. As can be seen from \cref{le:eformula}, there is a bijective correspondence between central idempotents \(e\) of \(kK\) such that \(\dim (kKe) = 1\) and \(1\)-dimensional representations of \(K\). Under this correspondence, the action of \(\widehat{K^{\mathrm{ab}}}\) as described in \cref{rk:actioncharacter} corresponds to the following transitive action of \(\widehat{K^{\mathrm{ab}}}\) on itself:
    \[\widehat{K^{\mathrm{ab}}} \times \widehat{K^{\mathrm{ab}}} \to \widehat{K^{\mathrm{ab}}} : (\nu, \chi) \mapsto \nu^{-1} \chi.\]
    Hence all central idempotents of \(kK\) such that \(\dim (kKe) = 1\) are in the same orbit, and give rise to isomorphic partial comodule algebras by \cref{th:redundancyI}. As a consequence, the only \(1\)-dimensional partial \(kG\)-comodules are indeed the ones coming from integrals of subgroups of \(G\) \eqref{eq:partial_integral} and their shifted versions \eqref{eq:partial_integral_trans}.
\end{proof}

Let us summarize the results of this section: subcentral idempotents \(e\) of \(kG\) are those idempotents such that \(ee_{(1)} \otimes e_{(2)} = e_{(1)}e \otimes e_{(2)}\) or, otherwise stated, the central idempotents of group algebras of subgroups \(K\leq G\) such that \(e \notin kN\) for any proper subgroup \(N < K\) (cf. \cref{le:ecentral}). For any such subcentral idempotent \(e\), the ideal \(I = kKe\) is a simple partial \(kG\)-comodule (by \cref{th:groupcase_simple}). Two of these partial comodules are isomorphic if and only if there is a multiplicative character \(\nu\) of \(K\) such that the algebra automorphism \(kK \to kK : g \mapsto \nu(g) g\) maps \(e\) to \(e'\).
These partial coactions can be shifted (similar to usual coactions of $kG$ which correspond to gradings of the group $G$). This can be done in \([G : K]\) non-isomorphic ways: \(I \to I \otimes kG : x \mapsto x_{(1)}e \otimes x_{(2)}\) can be shifted to \(I \to I \otimes kG : x \mapsto x_{(1)}e \otimes g_ix_{(2)}\) for each representative \(g_i\) of a \(K\)-coset in \(G\) (cf. \eqref{eq:Irhoi}). In conclusion, the subcentral idempotent \(e\) gives rise to \([G : K]\) simple partial comodules of dimension \(\dim (I)\).

\subsection{Examples} We will end this section with some examples and remarks that support \cref{conj:complete} and \cref{conj:Apar}.

\begin{example}
    Let's apply the construction on \(S_3 = \{1, \sigma, \sigma^2, \alpha, \sigma \alpha, \sigma^2 \alpha\}\), where \(\sigma^3 = \alpha^2 = 1, \alpha \sigma \alpha = \sigma^2\).
    Let \(k\) be a field of characteristic different from \(2\) and \(3\) that contains a primitive cubic root of unity \(\zeta\). The simple partial comodules of \(kS_3\) obtained from the construction are summarized in \cref{tab:S3}.

    \begin{table}
        \begin{tabular}{ccccc}
            Subgroup \(K\) & Subcentral idempotent \(e\) & Equivalent idempotents & \(\dim I\) & \([G : K]\) \\ \hline
            \(\{1\}\) & \(1\) & & \(1\) & \(6\) \\
            \(\{1, \alpha\}\) & \(\frac{1 + \alpha}{2}\) & \(\frac{1 - \alpha}{2}\) & \(1\) & \(3\) \\
            \(\{1, \sigma\alpha\}\) & \(\frac{1 + \sigma\alpha}{2}\) & \(\frac{1 - \sigma\alpha}{2}\) & \(1\) & \(3\) \\
            \(\{1, \sigma^2\alpha\}\) & \(\frac{1 + \sigma^2\alpha}{2}\) & \(\frac{1 - \sigma^2\alpha}{2}\) & \(1\) & \(3\) \\
            \(\{1, \sigma, \sigma^2\}\) & \(\frac{1 + \sigma + \sigma^2}{3}\) & \(\frac{1 + \zeta \sigma + \zeta^2\sigma^2}{3}\), \(\frac{1 + \zeta^2\sigma + \zeta\sigma^2}{3}\) & \(1\) & \(2\) \\
            \(\{1, \sigma, \sigma^2\}\) & \(\frac{2 - \sigma - \sigma^2}{3}\) & \(\frac{2 - \zeta \sigma - \zeta^2\sigma^2}{3}\), \(\frac{2 - \zeta^2\sigma - \zeta\sigma^2}{3}\) & \(2\) & \(2\) \\
            \(S_3\) & \(\frac{1 + \sigma + \sigma^2 + \alpha + \sigma \alpha + \sigma^2 \alpha}{6}\) & \(\frac{1 + \sigma + \sigma^2 - \alpha - \sigma \alpha - \sigma^2 \alpha}{6}\) & 1 & 1 \\
            \(S_3\) & \(\frac{5 - \sigma - \sigma^2 - \alpha - \sigma \alpha - \sigma^2 \alpha}{6}\) & \(\frac{5 - \sigma - \sigma^2 + \alpha + \sigma \alpha + \sigma^2 \alpha}{6}\) & 5 & 1
        \end{tabular}
        \caption{Summary of the simple partial comodules of \(kS_3\).}
        \label{tab:S3}
    \end{table}

    We have \(18\) one-dimensional partial comodules, \(2\) two-dimensional and \(1\) five-dimensional. It has been verified by computer (using SageMath) that the dimension of \((kS_3^*)_{par}\) is indeed equal to \(51 = 18 \cdot 1^2 + 2 \cdot 2^2 + 1 \cdot 5^2\), so this proves that our construction has found all simple partial comodules of \(S_3\) and that
    \[(kS_3^*)_{par} \cong k^{18} \times M_2(k)^2 \times M_5(k).\]
    In particular, \((kS_3^*)_{par}\) is semisimple. We also found by computer that 
    \[\dim A_{par}(kS_3^*) = 13,\]
    which is in agreement with \cref{conj:Apar}. Indeed, for \(e = \frac{2 - \sigma - \sigma^2}{3}\), the algebra \(Ae\) is isomorphic to \(k \times k\) and for \(e = \frac{5 - \sigma - \sigma^2 - \alpha - \sigma \alpha - \sigma^2 \alpha}{6}\), we have \(Ae \cong k \times M_2(k),\) so
    \[A_{par}(kS_3^*) \cong k^9 \times M_2(k). \qedhere\]
    \end{example}
    \begin{example}
Suppose that \(\mathrm{char}(k) \neq 2\) and that \(k\) containes a primitive fourth root of unity. 
   We calculated by computer that 
    \[\dim (kD_8^*)_{par} = \dim (kQ_8^*)_{par} = 180,\]
    where \(D_8\) is the dihedral group of \(8\) elements and \(Q_8\) is the quaternion group.
    By drawing a table as in the previous example and comparing with these dimensions, one shows that every simple partial comodule of \(D_8\) and \(Q_8\) is obtained by \cref{construction:general2}, and that
    \begin{align*}
        (kD_8^*)_{par} &\cong k^{35} \times M_2(k)^2 \times M_3(k)^7 \times M_5(k) \times M_7(k), \\
        (kQ_8^*)_{par} &\cong k^{19} \times M_2(k)^6 \times M_3(k)^7 \times M_5(k) \times M_7(k).
    \end{align*}
    For the subalgebra \(A_{par}\) we have
    \[A_{par}(kD_8^*) \cong A_{par}(kQ_8^*) \cong k^{28} \times M_2(k)^2,\]
    as can be predicted using \cref{conj:Apar}.
\end{example}

\begin{example}
    Let \(G\) be a finite abelian group and \(k\) an algebraically closed field such that its characteristic does not divide \(|G|\). Then \(kG \cong kG^*\) as Hopf algebras, and at the end of  \cref{se:construction} we showed that \cref{construction:general2} constructs every simple partial \(kG^*\)-comodule. Since \(kG_{par}\) is semisimple by \cite{DEP}, \eqref{eq:theta} is an isomorphism in this case.
\end{example}

\section{A non-commutative and non-cocommutative example: the Kac-Paljutkin algebra \texorpdfstring{\(\mathcal{A}\)}{A}}
\label{se:kac}

Let \(k\) be a field of characteristic different from \(2\) that contains an eighth root of unity \(\zeta\).

In this section we look at a more exotic example: the Kac-Paljutkin algebra \(\mathcal{A}\), which is the unique semisimple \(8\)-dimensional non-commutative non-cocommutative Hopf algebra over \(k\).
This algebra is described in \cite{Masuoka} and is studied in \cite{Shi}. It is generated as an algebra by \(x, y, z\) where 
	\begin{gather*}
		x^2 = y^2 = 1, \quad z^2 = \frac{1}{2}(1 + x + y - xy), \\ xy = yx, \quad xz = zy, \quad yz = zx.
	\end{gather*}
and the Hopf algebra structure is given by
\begin{align*}
	\Delta(x) &= x \otimes x, & \epsilon(x) &= 1, & S(x) &= x;\\
	\Delta(y) &= y \otimes y, & \epsilon(y) &= 1, & S(y) &= y;\\
	\Delta(z) &= \frac{1}{2}(1 \otimes 1 + 1 \otimes x + y \otimes 1 - y \otimes x)(z \otimes z), & \epsilon(z) &= 1, & S(z) &= z.
\end{align*}
A linear basis for \(\mathcal{A}\) is \(\{1, x, y, z, xy, xz, yz, xyz\}\) and \(t = \frac{1}{8}(1 + x+ y + xy + z + xz + yz + xyz)\) is a normalized integral. The group of grouplikes \(\{1, x, y, xy\}\) is isomorphic to the Klein four group. This way, we get immediately the following right coideal subalgebras (which are in fact Hopf subalgebras):
\[\langle 1 \rangle, \langle 1, x \rangle, \langle 1, y \rangle, \langle 1, xy \rangle, \langle 1, x, y, xy \rangle, \mathcal{A}.\]
There are exactly two more right coideal subalgebras, namely
\begin{align}
    S_1 &= \left\langle 1, xy, \frac{1-\zeta^2}{2}z + \frac{1 + \zeta^2}{2}xz, \frac{1 + \zeta^2}{2}yz + \frac{1 - \zeta^2}{2}xyz \right\rangle = \langle 1, xy, s, xys \rangle, \\
    S_2 &= \left\langle 1, xy, \frac{1 + \zeta^2}{2}z + \frac{1 - \zeta^2}{2}xz, \frac{1 - \zeta^2}{2}yz + \frac{1 + \zeta^2}{2}xyz \right\rangle = \langle 1, xy, \bar{s}, xy \bar{s} \rangle.
\end{align}
where \(s = \frac{1 - \zeta^2}{2}z + \frac{1 + \zeta^2}{2}xz\) and \(\bar{s} = \frac{1 + \zeta^2}{2}z + \frac{1 - \zeta^2}{2}xz\)

 The Kac-Paljutkin algebra \(\mathcal{A}\) is a self-dual Hopf algebra and is isomorphic to \(k^4 \times M_2(k)\). So as a coalgebra, \(\mathcal{A}\) is the direct sum of four copies of \(k\) and one \(2 \times 2\) comatrix coalgebra.

The subcentral idempotents \(e \in \mathcal{A}\) are those idempotents that are central in the right coideal subalgebra they generate. 
The right coideal subalgebra \(\langle 1 \rangle\) contains just the subcentral idempotent \(1\) and obviously in this case \(\bar{\mathcal{A}} = \mathcal{A}\). By taking for \(W\) the different simple \(\mathcal{A}\)-comodules, we find exactly the \(4\) one-dimensional and \(1\) two-dimensional simple global \(\mathcal{A}\)-comodules because \(W \cotensor^{\mathcal{A}} \mathcal{A} \cong W\).

\begin{table}[t]
    \begin{tabular}{ccccc}
        \(A\) & \(e\)  & \(\dim Ae\) & \makecell{\((\text{number} \cdot \text{dim}^2)\) \\ of partial comodules} \\ \hline
        \(\langle 1 \rangle\) & \(1\) & \(1\) & \(4 \cdot 1^2 + 1 \cdot 2^2\) \\
        \(\langle 1, x \rangle\) & \(\frac{1 + x}{2}\) & \(1\) & \(4 \cdot 1^2\) \\
        \(\langle 1, y \rangle\) & \(\frac{1 + y}{2}\) & \(1\) & \(4 \cdot 1^2\) \\
        \(\langle 1, xy \rangle\) & \(\frac{1 + xy}{2}\) & \(1\) & \(4 \cdot 1^2\) \\
        \(\langle 1, x, y, xy \rangle\) & \(\frac{1 + x+ y + xy}{4}\) & \(1\) & \(2 \cdot 1^2\) \\
        \(\langle 1, x, y, xy \rangle\) & \(\frac{3 - x - y - xy}{4}\) & \(3\) & \(2 \cdot 3^2\) \\
        \(S_1\) & \(\frac{1 + xy + s + xys}{4}\) & \(1\) & \(2 \cdot 1^2\) \\
        \(S_1\) & \(\frac{2 + (1 + \zeta) s + (1 - \zeta) xys}{4}\) & \(2\) & \(2 \cdot 2^2\) \\
        \(S_1\) & \(\frac{3 - xy - s - xys}{4}\) & \(3\) & \(2 \cdot 3^2\) \\
        \(S_2\) & \(\frac{1 + xy + \bar{s} + xy\bar{s}}{4}\) & \(1\) & \(2 \cdot 1^2\) \\
        \(S_2\) & \(\frac{2 + (1 - \zeta) \bar{s} + (1 + \zeta) xy\bar{s}}{4}\) & \(2\) & \(2 \cdot 2^2\) \\
        \(S_2\) & \(\frac{3 - xy - \bar{s} - xy\bar{s}}{4}\) & \(3\) & \(2 \cdot 3^2\) \\
        \(\mathcal{A}\) & \(\frac{1 + x + y + xy + z + xz + yz + xyz}{8}\) & \(1\) & \(1 \cdot 1^2\) \\
        \(\mathcal{A}\) & \(\frac{3 - x - y + 3 xy + z + xz + yz + xyz}{8}\) & \(3\) & \(1 \cdot 3^2\) \\
        \(\mathcal{A}\) & \(\frac{5 + x + y - 3 xy + z + xz + yz + xyz}{8}\) & \(5\) & \( 1 \cdot 5^2\) \\
        \(\mathcal{A}\) & \(\frac{7 - x - y - xy - z - xz - yz - xyz}{8}\) & \(7\) & \(1 \cdot 7^2\) \\
        &&Total&180
    \end{tabular}
    
    \caption{Summary of the simple partial comodules of \(\mathcal{A}\).}
     \label{tab:kac}
\end{table}

For every other choice of right coideal subalgebra, the quotient coalgebra \(\bar{\mathcal{A}}\) has a basis of grouplikes. The partial \(\mathcal{A}\)-comodules coming from \cref{construction:general2} (taking \(W\) a simple \(\bar{\mathcal{A}}\)-comodule) are summarized in \cref{tab:kac}. Equivalent subcentral idempotents (i.\,e.\ idempotents that produce isomorphic partial comodules via the construction) have been omitted. Each of the partial comodules obtained this way is simple. If the subcentral idempotent is a linear combination of grouplikes, this follows from the argument given in \cref{th:groupcase_simple}. In the remaining cases it can be verified explicitly. 

A computer check showed that indeed
\[\dim \mathcal{A}_{par} = 180,\]
so this means that every simple partial \(\mathcal{A}\)-comodule is obtained by the construction and that
\[\mathcal{A}_{par} \cong k^{23} \times M_2(k)^5 \times M_3(k)^7 \times M_5(k) \times M_7(k).\]

Moreover \(\dim A_{par}(\mathcal{A}) = 36\) and
\[A_{par}(\mathcal{A}) \cong k^{28} \times M_2(k)^2.\]

\subsection*{Acknowledgements}
JV would like to thank the FWB (F\'ed\'eration Wallonie-Bruxelles) for support through the ARC project ``From algebra to combinatorics, and back'' and the ``Fonds Thelam'' for support through the project ``Partial symmetries of Non-commutative spaces''.

\end{document}